\documentclass[a4paper,leqno]{amsart}
\usepackage{amsfonts}

\usepackage{amssymb,amsthm,amsmath,amsrefs,enumerate}
\usepackage{pstricks}
\usepackage{pst-plot}
\usepackage{pst-math}
\usepackage{pst-node}

\newcounter{mylisti} \newcounter{mylistii}
\newcounter{nest}
\newcommand{\defaultlabel}{}

\newcommand{\bm}{\ensuremath{\mathbb M}}
\newcommand{\bn}{\ensuremath{\mathbb N}}

\newcommand{\br}{\ensuremath{\mathbb R}}

\newcommand{\cC}{\ensuremath{\mathcal C}}

\newcommand{\cK}{\ensuremath{\mathcal K}}

\newcommand{\diam}{\operatorname{diam}}
\newcommand{\cdiam}{\operatorname{cdiam}}
\newcommand{\cdim}{\operatorname{cdim}}

\newcommand{\dist}{\ensuremath{\mathrm{dist}}}

\newcommand{\sep}{\ensuremath{\mathrm{sep}}}

\newcommand{\n}{\ensuremath{\bar{n}}}

\newcommand{\ds}{\displaystyle}

\newcommand{\codist}{\mathrm{codist}}
\newcommand{\lip}{\ensuremath{\mathrm{Lip}}}
\newcommand{\colip}{\ensuremath{\mathrm{coLip}}}

\newcommand{\co}{\mathrm{c}_0}

\newcommand{\xbl}{\ensuremath{\left(\sum_{n=1}^\infty\ell_{\infty}^n\right)_2}}

\newtheorem{theorem}{Theorem}[section]

\newtheorem{lemma}{Lemma}[section]
\newtheorem{proposition}{Proposition}[section]
\newtheorem{corollary}{Corollary}[section]
\newtheorem{claim}{Claim}[section]
\newtheorem{obs}{Observation}
\newtheorem{problem}{Problem}[section]

\theoremstyle{definition}

\newtheorem{defn}{Definition}[section]

\theoremstyle{remark}

\newtheorem{rem}{Remark}[section]

\begin{document}

\title[$(\beta)$-distortion of some infinite graphs]{$(\beta)$-distortion of some infinite graphs}

\author[F.~Baudier]{Florent~Baudier}
\address{Florent Baudier, Institut de Math\'ematiques Jussieu-Paris Rive Gauche, Universit\'e Pierre et Marie Curie, Paris, France and Department of Mathematics, Texas A\&M University, College Station, TX 77843-3368, USA (current)}
\email{flo.baudier@imj-prg.fr, florent@math.tamu.edu (current)}
\author[S.~Zhang]{Sheng~Zhang}
\address{Sheng Zhang, Department of Mathematics, Texas A\&M University, College Station, TX 77843-3368, USA}
\email{z1986s@math.tamu.edu}
\date{}

\thanks{The first author's research was partially supported by ANR-13-PDOC-0031, project NoLiGeA}
\thanks{The second author's research was partially supported by NSF DMS-1301604 and is part of his dissertation that is being prepared at Texas A\&M University under the direction of William B. Johnson}

\begin{abstract}
A distortion lower bound of $\Omega(\log(h)^{1/p})$ is proven for embedding the complete countably branching hyperbolic tree of height $h$ into a Banach space admitting an equivalent norm satisfying property $(\beta)$ of Rolewicz with modulus of power type $p\in(1,\infty)$ (in short property ($\beta_p$)). Also it is shown that a distortion lower bound of $\Omega(\ell^{1/p})$ is incurred when embedding the parasol graph with $\ell$ levels into a Banach space with an equivalent norm with property ($\beta_p$). The tightness of the lower bound for trees is shown adjusting a construction of Matou\v{s}ek to the case of infinite trees. It is also explained how our work unifies and extends a series of results about the stability under nonlinear quotients of the asymptotic structure of infinite-dimensional Banach spaces. Finally two other applications regarding metric characterizations of asymptotic properties of Banach spaces, and the finite determinacy of bi-Lipschitz embeddability problems are discussed.
\end{abstract}

\maketitle

\section{Introduction}
Let $(X,d_X)$ and $(Y,d_Y)$ be two metric spaces. $B_X(x,r)$ denotes the closed ball centered at $x\in X$ with radius $r>0$. A map $f\colon X\to Y$ is called a bi-Lipschitz embedding if it is one-to-one and both $f$ and $f^{-1}$ are Lipschitz. The {\it distortion} of $f$ is then defined as
$$ \dist(f):= \lip(f)\cdot\lip(f^{-1}):=\sup_{x\neq y \in X}\frac{d_Y(f(x),f(y))}{d_X(x,y)}.\sup_{x\neq y \in X}\frac{d_X(x,y)}{d_Y(f(x),f(y))}.$$
As usual $c_{Y}(X):=\inf\{\dist(f)\ |\ f\colon X\to Y \textrm{ is a bi-Lipschitz embedding}\}$ denotes the $Y$-distortion of $X$. If there is no bi-Lipschitz embedding from $X$ into $Y$ then we set $c_{Y}(X)=\infty$.

\medskip

In this article we study Banach spaces that satisfy a geometric property introduced in \cite{Rolewicz1987}, now known as property ($\beta$) of Rolewicz or simply property ($\beta$). The following is an equivalent definition of property ($\beta$) according to Kutzarova \cite{Kutzarova1991}.

\begin{defn}
A Banach space $X$ has property ($\beta$) if for any $\varepsilon>0$ there exists $\delta(\varepsilon)\in(0,1)$ so that for every element $x\in B_X$ and every sequence $(y_i)_{i=1}^{\infty}\subset B_X$ with $\sep(\{y_i\}_{i=1}^\infty)\geq\varepsilon$, there exists $i_0\in\bn$ such that
$$\left\|\dfrac{x-y_{i_0}}{2}\right\|\leq1-\delta(\varepsilon).$$
The separation constant of the sequence is defined by $\sep(\{y_i\}_{i=1}^\infty):=\inf\{\|y_n-y_m\|:n\neq m\}$. $B_X$ denotes the closed unit ball of $X$. A modulus for the property ($\beta$) was defined in \cite{AyerbeDominguezCutillas1994} as follows (we follow the notation of \cite{DKLR2014}):
$$\bar{\beta}_X(t)=1-\sup\left\{\inf_{i\ge 1}\left
\{\frac{\|x-y_i\|}{2}\right\}\colon x\in B_X, (y_i)_{i=1}^{\infty}\subset B_X,
\sep(\{y_i\}_{i=1}^\infty)\geq t\right\}.$$
\end{defn}
\noindent Note that $\bar{\beta}_X$ is a non-decreasing map defined on an interval $[0,a]$ where the constant $a\in[1,2]$ depends on the geometry of the Banach space $X$. The ($\beta$)-modulus of $X$ is said to have power type $p\in(1,\infty)$ with constant $\gamma\in(0,\infty)$ if $\bar{\beta}_X(t)\geq\gamma t^p$ for all $t\in[0,a]$. In that case we simply say that $X$ has property ($\beta_p$). The omitted definitions and notational conventions from Banach space theory can be found in \cite{Handbook}.

\medskip

The main concern in this article is the quantitative embedding theory of some infinite graphs into Banach spaces with property $(\beta)$. More precisely let $p\in(1,\infty)$, and define
$$\cC_{(\beta_p)}:=\{\textrm{Y has an equivalent norm with property $(\beta_p)$}\},$$
$$\cC_{(\beta)}:=\{\textrm{Y has an equivalent norm with property $(\beta)$}\}.$$
A typical example of a Banach space in $\cC_{(\beta_p)}$ is any reflexive $\ell_p$-sum of finite dimensional Banach spaces, e.g. $\ell_p$. Since property $(\beta)$ implies reflexivity neither $\ell_1$ nor $\co$ are in $\cC_{(\beta)}$. In \cite{DKLRpriv} it was shown that $X$ admits an equivalent norm with property $(\beta)$ if and only if $X$ admits an equivalent norm with property $(\beta_p)$ for some $p\in(1,\infty)$. In other words, $\bigcup_{p\in(1,\infty)}\cC_{(\beta_p)}=\cC_{(\beta)}$. The $(\beta_p)$-distortion of a metric space $X$ is the value of the parameter $c_{(\beta_p)}(X):=\inf\{c_{Y}(X): Y\in \cC_{(\beta_p)}\}$ that measures the best possible embedding of $X$ into a space with property $(\beta_p)$. It is worth mentioning that the problem of estimating the $(\beta_p)$-distortion of locally finite metric spaces has been essentially settled in \cite{BaudierLancien2008}, where it is shown for instance that every locally finite metric space admits a bi-Lipschitz embedding into $(\sum_{n=1}^\infty{\ell_{\infty}^n})_{\ell_p}$ with distortion at most $181$. In this article the $(\beta_p)$-distortion of some families of non-locally finite graphs is investigated. The content of this article is now described.

\smallskip

Section \ref{distortion} is devoted to obtaining lower bounds on the $(\beta_p)$-distortion of some families of non-locally finite graphs. In Section \ref{cbrembed} the family $(T^\omega_h)_{h=1}^\infty$ of complete countably branching hyperbolic trees is studied. It is proven that if $Y\in\cC_{(\beta_p)}$ then $c_{Y}(T^\omega_h)=\Omega(\log(h)^{1/p})$, i.e. $c_{Y}(T^\omega_h)\gtrsim \log(h)^{1/p}$ for $h$ big enough, where as usual the symbol $\gtrsim$ is meant to hide a constant depending eventually on the geometry of the receiving space $Y$ but not on $h$. The proof combines an asymptotic version of the prong bending lemma from \cite{Kloeckner2014} (see also \cite{Matousek1999} for a similar argument) and a self-improvement argument \`a la Johnson and Schechtman \cite{JohnsonSchechtman2009} which was elegantly implemented in the case of binary trees by Kloeckner \cite{Kloeckner2014}. A similar lower bound is shown in Section \ref{parasols} for the family of parasol graphs introduced by Dilworth, Kutzarova, and Randrianarivony \cite{DKR2014}.

\smallskip

The optimality of the lower bound for trees is discussed in Section \ref{optimal}. Adjusting a construction of Matou\v{s}ek \cite{Matousek1999} to the case of infinite weighted trees, an upper bound $c_{\ell_p}(T)=O(\log(\kappa^*(T))^{1/p})$ (i.e. $c_{\ell_p}(T)\lesssim\log(\kappa^*(T))^{1/p}$) is proved where $\kappa^*(T)$ is a coloring parameter related to the combinatorial structure of $T$. The relationship between the caterpillar-coloring parameter $\kappa^*(T)$ and the strong-coloring parameter $\delta^*(T)$ introduced by Lee, Naor, and Peres \cite{LeeNaorPeres2009} is discussed.
\smallskip

In the last section several applications of the present work is gathered. Regarding the stability of the asymptotic structure of infinite-dimensional Banach spaces under nonlinear quotients, it is shown how this work unifies, and extends, a series of results from \cite{LimaR2012}, \cite{DKLR2014}, \cite{Zhang2015}, and \cite{DKR2014}. The quantitative approach devised in this article takes full advantage of the simple observation that the quantitative theories of bi-Lipschitz embeddings and Lipschitz quotients coincide, in a precise sense, for trees. For instance, an elementary proof of the fact that $\ell_q$ is not a Lipschitz quotient of a subset of $\ell_p$ when $q>p>1$ follows from the work presented here. New insights are also given regarding the metric characterization of two equivalent classes of Banach spaces: the class of reflexive spaces admitting an asymptotically uniformly convex equivalent norm and an asymptotically uniformly smooth equivalent norm, and the class of spaces admitting an equivalent norm with property $(\beta)$. Finally the finite determinacy of bi-Lipschitz embeddability problems is discussed and it is shown that a theorem of Ostrovskii \cite{Ostrovskii2012} cannot be extended to non-locally finite graphs.
\section{Embeddability into spaces with property $(\beta)$}\label{distortion}

Recall that a weighted connected simple graph is a connected graph $G=(V,E)$ with no multiple edges or self-loop, equipped with a positive weight function $w\colon E\to(0,\infty)$. A graph is unweighted if every edge has unit weight. $G$ will always be equipped with its canonical metric $\rho_G$ on its set of vertices, where $$\rho_G(x,y):=\inf\{\sum_{e\in P}w(e)\colon \textrm{$P$ is a path connecting $x$ to $y$}\}.$$ A weighted tree is an acyclic weighted connected simple graph. In a tree two vertices are connected by a unique path and a leaf is a vertex of degree 1. For technical reasons we shall work with rooted trees. When we root a tree at an arbitrary vertex $r$ the ancestor-descendant relationship between pairs of vertices is then well defined. The height of a vertex $x$ of a rooted tree $T$, denoted by $h(x)$, is the number of edges separating $x$ from the root. The height of a rooted tree $T$ is then defined by $h(T):=\sup_{x\in T}h(x)$. The last common ancestor (in the ancestor-descendant relationship) of two vertices $x$ and $y$ is denoted by $lca(x,y)$. With this notation the canonical graph distance on an unweighted rooted tree is given explicitly by $$\rho_T(x,y)=h(x)+h(y)-2h(lca(x,y))=\rho_T(x,lca(x,y))+\rho_T(lca(x,y),y).$$ For a positive integer $h$, $T^\omega_h$ denotes the unweighted complete countably branching rooted tree of height $h$, while $T^\omega_\omega$ will be the unweighted complete countably branching rooted tree of infinite height.

\subsection{Complete countably branching trees}\label{cbrembed}
$K_{\omega,1}$ denotes the star graph with countably many branches, i.e. the bipartite graph that has a partition into exactly two classes, one consisting of a singleton called the center, the other consisting of countably many vertices called the leaves. In the sequel $b$ will denote the center. An arbitrary leaf, denoted by $r$, is chosen, and a labeling $(t_i)_{i=1}^\infty$ of the (countably many) remaining leaves is fixed. With this labeling in mind $K_{\omega,1}$ can be seen as an umbel with countably many pedicels, where $r$ stands for root, $b$ for the branching point on the stem, and $(t_i)_{i=1}^\infty$ is a labeling of the tips of the pedicels. As usual $K_{\omega,1}$ is equipped with the shortest path metric. The next lemma says that if the umbel is embedded into a space with property $(\beta)$ then at least one pedicel has to bend towards the root, and the distance from its tip to the root is shorter than expected. It can be seen as an asymptotic analogue of Lemma $2$ in \cite{Kloeckner2014}.

\begin{lemma}[Umbel pedicel bending lemma]\label{umbel} Let $Y$ be a Banach space with property ($\beta)$. Then for every bi-Lipschitz embedding $f\colon K_{\omega,1}\to Y$ there exists $i_0\in\bn$ such that
\begin{equation}\label{tipbending}
\left\|f(r)-f(t_{i_0})\right\|\le 2\lip(f)\left(1-\bar{\beta}_Y\left(\frac{2}{\dist(f)}\right)\right).
\end{equation}
\end{lemma}

\begin{proof}
One may assume after an appropriate translation that $f(b)=0$. Let $x=\ds\frac{f(r)}{\lip(f)}$ and $\ds y_i=\frac{f(t_i)}{\lip(f)}$. Clearly $\|x\|\le 1,$ $\|y_i\|\le1$ for all $i\in\bn$, and for $n\neq m$,
$$\|y_n-y_m\|\ge \ds\frac{2}{\dist(f)}>0.$$
Since the norm of $Y$ satisfies property $(\beta)$ there exists $i_0\in\bn$ such that $$\left\|\frac{x-y_{i_0}}{2}\right\|\le1-\bar{\beta}_Y\left(\frac{2}{\dist(f)}\right),$$ and hence the result follows.
\end{proof}

\begin{rem}
Note that the conclusion of Lemma \ref{umbel} can be strengthened. Since \eqref{tipbending} holds for all but finitely many $i$'s, there exists an infinite subset $\bm\subset\bn$ such that $$\max\left\{\sup_{i\in\bm}\left\|f(r)-f(t_{i})\right\|;\sup_{i\neq j\in\bm}\left\|f(t_i)-f(t_{j})\right\|\right\}\le2\lip(f)\left(1-\bar{\beta}_Y\left(\frac{2}{\dist(f)}\right)\right).$$
\end{rem}
The next proposition is a self-improvement argument \`a la Johnson and Schechtman \cite{JohnsonSchechtman2009}. It is shown that if the countably branching tree of a certain height embeds into a Banach space with property ($\beta$) then the countably branching tree of roughly half the height embeds as well, but with a slightly better distortion. The vertex set of $T^\omega_h$ can be naturally labelled by elements in $\bigcup_{i=0}^{h}\bn^h$, where by convention $\bn^0=\emptyset$ is the label assigned to the root. The notation $\n=(n_1,\dots,n_r)$, for some $r\le h$, designates a generic vertex of $T^\omega_h$.

\begin{proposition}\label{treeimprovement}
Let $Y$ be a Banach space with property ($\beta$). Let $k\in \bn$, and assume that  $T^{\omega}_{2^k}$ bi-Lipschitzly embeds into $Y$ with distortion $D$. Then $T^{\omega}_{2^{k-1}}$ bi-Lipschitzly embeds into $Y$ with distortion at most $D(1-\bar{\beta}_Y(\frac{2}{D}))$.
\end{proposition}

\begin{proof}
Let $f\colon T^\omega_{2^k}\to Y$ be a bi-Lipschitz embedding with distortion $D$. In order to define an embedding of $T^\omega_{2^{k-1}}$ into $Y$ one selects vertices located at even heights following a simple procedure. The set of all elements of height at most $2$ in the tree $T^\omega_{2^k}$ can be seen as being formed by  countably many umbels. For every $n_1\in\bn$, consider the umbel whose root is the vertex $\emptyset$ and whose branching point is the vertex $(n_1)\in T^\omega_{2^k}$. By Lemma \ref{umbel} there is a vertex located at level $2$ which is ``close'' to the root of the umbel, i.e. there exists $t_{(n_1)}\in\bn$ such that $$\|f(\emptyset)-f((n_1,t_{(n_1)}))\|\le 2\lip(f)\left(1-\bar{\beta}_Y\left(\frac{2}{D}\right)\right).$$ For every vertex $(n_1,t_{(n_1)})$ as above, and for every $n_2\in\bn$, consider the umbel whose root is the vertex $(n_1,t_{(n_1)})$, and whose branching point is the vertex $(n_1,t_{(n_1)},n_2)$. Again select using Lemma \ref{umbel}, a level-4 vertex that is the tip of the bending pedicel, i.e. there exists $t_{(n_1,n_2)}\in\bn$ such that $$\|f((n_1,t_{(n_1)}))-f((n_1,t_{(n_1)},n_2,t_{(n_1,n_2)}))\|\le 2\lip(f)\left(1-\bar{\beta}_Y\left(\frac{2}{D}\right)\right).$$
Repeat this procedure until vertices located in the set of leaves of the tree are selected. To summarize we have chosen a collection of integers $(t_{\n})_{\n\in T^\omega_{2^{k-1}}}$ such that for every $\n=(n_1,\dots,n_r)\in T^\omega_{2^{k-1}}$ one has
\begin{align*}
\|f((n_1,t_{(n_1)},\dots,n_{r-1},t_{(n_1,\dots,n_{r-1})}))-f((n_1,t_{(n_1)},\dots,n_{r},t_{(n_1,\dots,n_{r})}))\|\le
\end{align*}
\begin{align*}
2\lip(f)\left(1-\bar{\beta}_Y\left(\frac{2}{D}\right)\right).
\end{align*}
Finally define
\begin{align*}
g\colon T^\omega_{2^{k-1}}&\to Y, \hspace{2mm}
\n=(n_1,\dots,n_r)\mapsto \frac{f((n_1,t_{(n_1)},\dots,n_{r},t_{(n_1,\dots,n_{r})}))}{2}
\end{align*}
and $g(\emptyset)=\frac{1}{2}f(\emptyset)$. Since for a graph it is sufficient to consider adjacent vertices to estimate the Lipschitz constant, one can easily check that $\dist(g)\le D(1-\bar{\beta}_Y(\frac{2}{D}))$.

\end{proof}

\begin{theorem}\label{treedist} Let $Y$ be a Banach space admitting an equivalent norm with property $(\beta)$. Then $\sup_{h\ge1}c_Y(T^\omega_h)=\infty$.

\smallskip

\noindent In particular, if $Y$ is a Banach space with property $(\beta_p)$ with $p\in(1,\infty)$ and constant $\gamma:=\gamma(Y)>0$, then $c_Y(T^\omega_h)\ge 2\gamma^{\frac{1}{p}}\log(\frac{h}{2})^{\frac{1}{p}}$.
\end{theorem}

\begin{proof}
Let $D_h:=c_Y(T^\omega_h)$ in the sequel, and assume that $\sup_{h\ge 1}D_h=D\in(0,\infty)$. Assume without loss of generality that $(D_h)_{h\ge 1}$ is a converging sequence. According to Proposition \ref{treeimprovement}, if $Y$ has property ($\beta$) then for every $k\ge 1$ one has $D_{2^{k-1}}\le D_{2^k}(1-\bar{\beta}_Y(\frac{2}{D}))$; taking the limit in $k$ gives a contradiction. Suppose that $Y$ has ($\beta_p$) with $p\in(1,\infty)$ and constant $\gamma:=\gamma(Y)>0$. It follows from Proposition \ref{treeimprovement}, that for all $j\le k$ one has  $D_{2^{j-1}}\le D_{2^j}(1-\frac{2^p\gamma}{D_{2^j}^{p}})$, where $k$ is such that $2^k\le h<2^{k+1}$. Therefore $D_{2^{j}}-D_{2^{j-1}}\ge \frac{2^p\gamma}{D_{2^j}^{p-1}}$ and $$D_h\ge D_{2^k}\ge 2^p\gamma\sum_{j=1}^k D_{2^j}^{1-p}+D_1\ge 2^p\gamma k D_{h}^{1-p}.$$ The conclusion follows easily.
\end{proof}
\subsection{Parasol graphs}\label{parasols}
In this section we consider again the graph $K_{\omega,1}$ in its umbel configuration. However, an extra vertex is introduced, denoted by $s$, and $s$ is connected to each of the tips of the pedicels by a single edge. $P^\omega_1$ denotes the new graph obtained, which looks like a parasol. Lemma \ref{parasol} is a simple consequence of Lemma \ref{umbel}.

\begin{lemma}[Parasol top bending lemma]\label{parasol} Let $Y$ be a Banach space with property ($\beta)$. Then for every bi-Lipschitz embedding $f\colon P^\omega_1\to Y$ one has $$\left\|f(r)-f(s)\right\|\le 3\lip(f)\left(1-\frac{2}{3}\bar{\beta}_Y\left(\frac{2}{\dist(f)}\right)\right).$$
\end{lemma}

\begin{proof}

The inequality follows from Lemma \ref{umbel}, the fact that $\|f(t_{i_0})-f(s)\|\le\lip(f)$, and the triangle inequality.
\end{proof}

The parasol graph $P^{\omega}_{\ell}$ can be defined as in \cite{DKR2014} using a fractal-like procedure. The parasol graph of level $1$ is just the graph $P^{\omega}_{1}$. The parasol graph of level $2$ is simply the graph obtained by replacing each edge in $P^{\omega}_{1}$ with a copy of $P^{\omega}_{1}$. Proceeding recursively $P^{\omega}_{\ell}$, the parasol graph of level $\ell$, is obtained by replacing each edge in $P^{\omega}_{\ell-1}$ by a copy of $P^{\omega}_{1}$. Proposition \ref{parasolimprovement} below is the analogue of Proposition \ref{treeimprovement} and can be used in the same way to prove Theorem \ref{parasoldist}. Indeed, we simply select the root vertex and the summit vertex in each copy of $P^{\omega}_{1}$ constituting $P^{\omega}_{\ell}$ to obtain a rescaled isometric copy (scaling factor of $3$) of $P^{\omega}_{\ell-1}$, and we iterate $\ell$ times to obtain Theorem \ref{parasoldist}. The details are left to the reader.

\begin{proposition}\label{parasolimprovement}
Let $Y$ be a Banach space with property ($\beta$). Let $\ell\in\bn$, and assume that  $P^{\omega}_{\ell}$ bi-Lipschitzly embeds into $Y$ with distortion $D$. Then $P^{\omega}_{\ell-1}$ bi-Lipschitzly embeds into $Y$ with distortion at most $D(1-\frac{2}{3}\bar{\beta}_Y(\frac{2}{D}))$.
\end{proposition}

\begin{theorem}\label{parasoldist}
Let $Y$ be a Banach space admitting an equivalent norm with property $(\beta)$. Then $\sup_{\ell\ge1}c_Y(P^\omega_\ell)=\infty$.

\smallskip

\noindent In particular, if $Y$ is a Banach space with property $(\beta_p)$ with $p\in(1,\infty)$ and constant $\gamma:=\gamma(Y)>0$, then $c_Y(P^\omega_\ell)\ge 2(\frac{2\gamma}{3})^{\frac{1}{p}}\ell^{\frac{1}{p}}$.
\end{theorem}

\section{Lebesgue distortion of infinite weighted trees}\label{optimal}
Bourgain \cite{Bourgain1986a} gave a simple embedding of the complete hyperbolic binary tree of height $h$ into $\ell_2$ with distortion $O(\sqrt{\log(h)})$. Bourgain's construction can be easily adjusted, and generalized to an arbitrary unweighted tree $T$, to give an embedding into $\ell_p$ with distortion $O_p(\log(\diam(T))^{1/p})$. The notation $O_p(\cdot)$ means that $c_{\ell_p}(T)\le K(\log(\diam T)^{1/p})$ where the constant $K\in(0,\infty)$ depends only on $p$ and not on $T$. This upper bound is already sufficient to show that the lower bound in Theorem \ref{treedist} is tight and optimal up to constant factors. The case of weighted trees is significantly more complicated (even for finite trees). Also, it is clear that an upper bound involving the diameter is not optimal, since an infinite path embeds isometrically into the real line. Both issues were treated in \cite{LinialMagenSaks1998} and \cite{Matousek1999} for finite trees. Parts of both arguments rely (implicitly in \cite{LinialMagenSaks1998} and explicitly in \cite{Matousek1999}) on a combinatorial parameter associated to a combinatorial tree, namely, the caterpillar dimension. The caterpillar dimension of a tree is related to its combinatorial structure and does not take into account the edge weights. Linial, Magen, and Saks showed that the Euclidean distortion of every \textit{finite weighted tree} $T$ with $l(T)$ leaves is bounded above by $O(\log\log(l(T))$. Let $\cdim(T)$ denote the caterpillar dimension of a tree $T$. By induction it is fairly easy to show that for every finite tree $T$ one has $\cdim(T)=O(\log(l(T)))$ (cf. \cite{LinialMagenSaks1998} or \cite{Matousek1999}). Theorem \ref{cdimfinite} was proved by Matou\v{s}ek \cite{Matousek1999}.

\begin{theorem}[\cite{Matousek1999}]\label{cdimfinite} For any $p\in(1,\infty)$ and for any finite weighted tree $T$, there exists an embedding of $T$ into $\ell_p$ with distortion $O_p(\log(\cdim(T))^{\min\{\frac{1}{2};\frac{1}{p}\}})$.
\end{theorem}
The construction of Matou\v{s}ek's embedding is very clever and delicate. It is mentioned in \cite{Matousek1999} that the caterpillar dimension can be defined for infinite trees, and that Theorem \ref{cdimfinite} holds for infinite weighted trees with a finite caterpillar dimension. The later statement can be misleading when compared to Theorem \ref{treedist} since $\cdim(T_h^\omega)=h$ for the natural extension of the caterpillar dimension to the infinitary setting. It is very likely that the mention was meant to say that Theorem \ref{cdim} below, which was proved but not stated in \cite{Matousek1999}, holds for infinite weighted trees with a finite caterpillar dimension.
\begin{theorem}[\cite{Matousek1999}]\label{cdim} For any $p\in(1,\infty)$ and for any finite weighted tree $T$ there exists a set $I$ and an embedding of $T$ into $\ell_p(I)$ with distortion $O_p(\log(\cdim(T))^{\frac{1}{p}})$.
\end{theorem}
Theorem \ref{cdimfinite} follows from Theorem \ref{cdim} by classical local arguments from Banach space theory, that can fail when applied to non-locally finite spaces (cf. Section \ref{determinacy}). Continuing the investigation from \cite{GuKL2003}, Lee, Naor, and Peres \cite{LeeNaorPeres2009} improved the upper  bound in Theorem \ref{cdimfinite}, using a coloring parameter that takes into account the edge weights. In this section, the necessary modifications to prove the infinitary version of Theorem \ref{cdim} are given. The exposition of the proof follows closely the one in \cite{LeeNaorPeres2009} and is based on a graph coloring approach.

\medskip

Let $T=(V,E)$ be a tree. Assume that $T$ is rooted at an arbitrary vertex. To extend the notion of caterpillar decomposition to the infinitary setting one needs to consider more general paths than just root-leaf paths. A \textit{root-leafend path} is either a path from the root to a leaf or a ray starting at the root. A \textit{monotone path} is a connected subset of some root-leafend path. Let $\cC$ be a set with $|\cC|\ge |E|$.  An \textit{edge coloring} of $T$ is a map $\chi\colon E\to \cC$. A coloring is \textit{monotone} if for every $c\in\cC$ the color class $\chi^{-1}(c)$ is a monotone path.
\begin{defn}[$\kappa$-caterpillar coloring] We say that a coloring is $\kappa$-caterpillar if it is monotone, and if every root-leafend path contains at most $\kappa$ distinct color classes. Let $\kappa_*(T):=\inf\{\kappa \colon $T$ \textrm{ admits a $\kappa$-caterpillar coloring}\}$
\end{defn}

Since a finite tree does not have rays, every root-leafend path is actually a root-leaf path, and a $\kappa$-caterpillar coloring of a finite tree is a caterpillar decomposition with width $\kappa$ in the terminology of Gupta \cite{Gupta2000}. It follows that $\kappa_*(T)=\cdim(T)$ for finite trees, and the coloring parameter $\kappa_*(T)$ can be used to define the caterpillar dimension of infinite trees. It is easy to see that $\kappa_*(T^\omega_h)=h$. Indeed, if one considers the monotone coloring, where no two distinct edges can be colored with the same color, then every root-leafend path intersects with exactly $h$ distinct color classes, and this cannot be improved. Theorem \ref{catdistortion} is the extension of Theorem \ref{cdim} to the infinitary setting.

\begin{theorem}\label{catdistortion}Let $p\in(1,\infty)$.
For any weighted tree $T$ there exists a set $I$ and an embedding of $T$ into $\ell_p(I)$ with distortion $O_p(\log(\kappa_*(T))^{\frac{1}{p}})$.
\end{theorem}

\begin{proof}
Let $T:=(V,E,w)$ be a weighted tree and assume that $T$ admits a coloring $\chi\colon E\to \cC$ that is $\kappa$-caterpillar with respect to some root $r$. Denote by $(e_c)_{c\in\cC}$ the canonical basis in $\ell_p(\cC)$. For $x,y\in V$, $P(x,y)\subset E$ denotes the unique path from $x$ to $y$. For a vertex $x$ in $T$ denote by $(c_1(x),\dots,c_{m(x)}(x))\in\cC^{m(x)}$ the sequence of color classes encountered on the path from the root to $x$. There are at most $\kappa$ such color classes. The distance that the color class $c_j(x)$ contributes to the path from the root to $x$ is $$\ell_j(x):=\sum_{\underset{e\in P(r,x)}{\chi(e)=c_j(x)}}w(e).$$
For a real number $\alpha$ we use the notation $\alpha^+:=\max\{0,\alpha\}$. The embedding $f$ from $T$ into $\ell_p(\cC)$
is given by $$f(x)=\sum_{i=1}^{m(x)}\ell_i(x)^{1/p}s_i(x)^{(p-1)/p} e_{c_i(x)},$$ where for $1\le i\le m(x)$,
$$s_i(x):=\sum_{j=i}^{m(x)}\left(\ell_j(x)-\frac{\ell_i(x)}{2\kappa}\right)^+.$$

The following observation is easy.
\begin{obs}\label{observ} For every $x,y\in T$ and every $i\in\{1,\dots,m(x)\}$
\begin{equation*}
|s_i(x)-s_i(y)|\le \rho_T(x,y).
\end{equation*}
\end{obs}
\begin{claim}\label{lipinv}
$\lip(f^{-1})\le 96$.
\end{claim}

\begin{proof}[Proof of Claim \ref{lipinv}:]\renewcommand{\qedsymbol}{}
Fix $x,y\in V$, $x\neq y$. Assume without loss of generality that $\ell_i(y)=\ell_i(x)$ for $i\in\{1,\dots,j-1\}$ and $\ell_{j+1}(x)\ge \ell_{j+1}(y)$. With this notation
\begin{equation}\label{distance}
\rho_T(x,y)=\ell_{j}(x)-\ell_{j}(y)+\sum_{i=j+1}^{m(x)}\ell_i(x)+\sum_{i=j+1}^{m(y)}\ell_i(y).
\end{equation}
The following simple observation will be used repeatedly.
\begin{obs}\label{genobs}
For all $x\in T$ and $i\in\{1,\dots,m(x)\}$, $s_i(x)\ge\ds\sum_{i=1}^{m(x)}\frac{\ell_i(x)}{2}$.
\end{obs}
On the other hand,
\begin{align*}
\|f(x)-f(y)\|_p^p\ge|[\ell_j(x)]^{1/p}&[s_j(x)]^{(p-1)/p}-[\ell_j(y)]^{1/p}[s_j(y)]^{(p-1)/p}|^p\\
&+\sum_{i=j+1}^{m(x)}\ell_i(x)[s_i(x)]^{p-1}+\sum_{i=j+1}^{m(y)}\ell_i(y)[s_i(y)]^{p-1}.
\end{align*}
Using Observation \ref{genobs} and the non-decreasingness of $t\mapsto t^{p-1}$ one has:

\begin{align*}
\sum_{i=j+1}^{m(x)}\ell_i(x)[s_i(x)]^{p-1}&\ge \frac{1}{2^{p-1}}\sum_{i=j+1}^{m(x)}\ell_i(x)\left(\sum_{h=i}^{m(x)}\ell_i(x)\right)^{p-1}\\
&\ge \frac{1}{2^{p-1}}\sum_{i=j+1}^{m(x)}\int_{\ell_i(x)+\dots+\ell_{m(x)}(x)}^{\ell_i(x)+\dots+\ell_{m(x)}(x)}t^{p-1}dt\\
&= \frac{1}{2^{p-1}}\int_{0}^{\ell_{j+1}(x)+\dots+\ell_{m(x)}(x)}t^{p-1}dt\\
&=\frac{1}{p2^{p-1}}\left(\sum_{i=j+1}^{m(x)}\ell_i(x)\right)^p.
\end{align*}

Similarly,

\begin{align*}
\sum_{i=j+1}^{m(y)}\ell_i(y)[s_i(y)]^{p-1}\ge\frac{1}{p2^{p-1}}\left(\sum_{i=j+1}^{m(y)}\ell_i(y)\right)^p.
\end{align*}

It remains to consider two cases:

\medskip

\noindent\textbf{Case 1.} $\frac{\ell_j(x)-\ell_j(y)}{3}\le\sum_{i=j+1}^{m(y)}\ell_{i}(y)$. In this case, it follows from \eqref{distance} that
\begin{align*}
\rho_T(x,y)^p&\le 4^p\left(\sum_{i=j+1}^{m(x)}\ell_i(x)+\sum_{i=j+1}^{m(y)}\ell_i(y)\right)^p\\
&\le 4^p\cdot2^{p-1}\left(\left(\sum_{i=j+1}^{m(x)}\ell_i(x)\right)^p+\left(\sum_{i=j+1}^{m(y)}\ell_i(y)\right)^p\right)\\
&\le p4^{2p-1}\cdot\|f(x)-f(y)\|_p^p\\
&\le 32^p\cdot\|f(x)-f(y)\|_p^p.
\end{align*}

\noindent\textbf{Case 2.} $\frac{\ell_j(x)-\ell_j(y)}{3}>\sum_{i=j+1}^{m(y)}\ell_i(y)$. In this case observe that
\begin{align*}
s_j(y)\le(1-\frac{1}{2\kappa})\ell_j(y)+\sum_{i=j+1}^{m(y)}\ell_i(y)\le\frac{2\kappa-1}{2\kappa}\ell_j(y)+\frac{\ell_j(x)-\ell_j(y)}{3}.
\end{align*}
Let $\cK=\frac{2\kappa-1}{2\kappa}\in[\frac{1}{2},1)$. Since $s_j(x)\ge (1-\frac{1}{2\kappa})\ell_j(x)=\cK\cdot\ell_j(x)$, one has
\begin{align*}
|[\ell_j(x)]^{1/p}&[s_j(x)]^{(p-1)/p}-[\ell_j(y)]^{1/p}[s_j(y)]^{(p-1)/p}|\\
& \ge \cK^{(p-1)/p}\ell_j(x)-\cK^{(p-1)/p}\ell_j(y)\left(1+\frac{\ell_j(x)-\ell_j(y)}{3\cK\cdot\ell_j(y)}\right)^{\frac{p-1}{p}}\\
& \ge \cK^{(p-1)/p}\ell_j(x)-\cK^{(p-1)/p}\ell_j(y)\left(1+\frac{\ell_j(x)-\ell_j(y)}{3\cK\cdot\ell_j(y)}\right)\\
& \ge (\ell_j(x)-\ell_j(y))\cK^{(p-1)/p}(1-\frac{1}{3\cK})\\
& \ge (\ell_j(x)-\ell_j(y))\cK(1-\frac{1}{3\cK})\\
& \ge\frac{\ell_j(x)-\ell_j(y)}{6}.
\end{align*}
It follows that
\begin{align*}
\|&f(x)-f(y)\|_p^p\\
&\ge \frac{1}{6^p}(\ell_j(x)-\ell_j(y))^p+\frac{1}{p4^{2p-1}}\left(\sum_{i=j+1}^{m(x)}\ell_i(x)\right)^p+\frac{1}{p4^{2p-1}}\left(\sum_{i=j+1}^{m(y)}\ell_i(y)\right)^p\\
&\ge \frac{1}{32^p\cdot3^{p-1}}\left[(\ell_j(x)-\ell_j(y))^p+\left(\sum_{i=j+1}^{m(x)}\ell_i(x)\right)^p+\left(\sum_{i=j+1}^{m(y)}\ell_i(y)\right)^p\right]\\
&\ge \frac{1}{96^p}\left[\ell_j(x)-\ell_j(y)+\sum_{i=j+1}^{m(x)}\ell_i(x)+\sum_{i=j+1}^{m(y)}\ell_i(y)\right]^p\\
&\ge \left(\frac{\rho_T(x,y)}{96}\right)^p.
\end{align*}
\end{proof}

\begin{claim}\label{lip}
$\lip(f)\le (6\log_2(2\kappa))^{\frac{1}{p}}$.
\end{claim}

\begin{proof}[Proof of Claim \ref{lip}:]\renewcommand{\qedsymbol}{}
It is sufficient to consider adjacent vertices $x$ and $y$, and one shall assume without loss of generality that $y$ is the vertex that is the farthest from the root. In this case $c_1(y)=c_1(x),\dots,c_{m(y)-1}(y)=c_{m(y)-1}(x)$, and $m(y)\in\{m(x), m(x)+1\}$. It follows that $\ell_i(y)=\ell_i(x)$ for $i\in\{1,\dots,m(y)-1\}$. When the edge connecting $x$ and $y$ is of a different color than $c_{m(y)-1}(y)$ set $\ell_{m(y)}(x)=0$ as a matter of convenience. Then,
\begin{align*}
\|f(x)-&f(y)\|_p^p=\|\sum_{i=1}^{m(y)}\ell_i(x)^{1/p}s_i(x)^{(p-1)/p} e_{c_i(x)}-\sum_{i=1}^{m(y)}\ell_i(y)^{1/p}s_i(y)^{(p-1)/p} e_{c_i(y)}\|_p^p\\
                  &\le\sum_{i=1}^{m(y)-1}\ell_i(y)|[s_i(x)]^{(p-1)/p}-[s_i(y)]^{(p-1)/p}|^p+\rho_{T}(x,y)^p\left(1-\frac{1}{2\kappa}\right)^{p-1}.
\end{align*}
The following observations are crucial in the sequel.
\begin{obs}\label{obsadj}
Let $x$ and $y$ be adjacent vertices such that $y$ is the farthest vertex from the root. Then, for all $i\in J:=\{j\in\{1,\dots,m(y)-1\}\colon s_j(y)\neq s_j(x)\}$
\begin{equation*}
s_{i}(y)\ge s_i(x)\textrm{ and }\ell_{m(y)}(y)>\frac{\ell_i(y)}{2\kappa}.
\end{equation*}
\end{obs}
It follows from the first inequality in Observation \ref{obsadj} and the inequality $a^s-b^s\le a^{s-1}(a-b)$ that holds for every $s\in[0,1]$ and $a>b>0$, that $$|[s_i(x)]^{(p-1)/p}-[s_i(y)]^{(p-1)/p}|\le \frac{s_i(y)-s_i(x)}{[s_i(y)]^{1/p}}.$$
By Observation \ref{observ} and Observation \ref{genobs} one gets
\begin{align*}
\sum_{i=1}^{m(y)-1}\ell_i(y)&|[s_i(x)]^{(p-1)/p}-[s_i(y)]^{(p-1)/p}|^p \le \sum_{i=1}^{m(y)-1}\ell_i(y)\frac{|s_i(y)-s_i(x)|^p}{s_i(y)}\\
& \le \rho_T(x,y)^p\sum_{i\in J}\frac{\ell_i(y)}{s_i(y)}\\
&\le 2\rho_T(x,y)^p\sum_{i\in J}\frac{\ell_i(y)}{\ell_i(y)+\dots+\ell_{m(y)-1}(y)+\ell_{m(y)}(y)}.
\end{align*}
Since $t\mapsto t+1$ is decreasing, for every  $x_1,\dots,x_k>0$ one has
\begin{align*}
\sum_{n=1}^k\frac{x_n}{x_n+\dots+x_{n+1}+\dots+x_k+1}\le \int_{0}^{x_1+\dots+x_{k+1}}\frac{dt}{t+1}
\le\log(x_1+\dots+x_{k}+1),
\end{align*}
and it follows that
\begin{align*}
\sum_{i\in J}\frac{\ell_i(y)}{\sum_{j=i}^{m(y)}\ell_j(y)}&\le\sum_{i\in J}\frac{\ell_i(y)/\ell_{m(y)}(y)}{\ds\sum_{j\in J;j\ge i}\ell_j(y)/\ell_{m(y)}(y))+1}\\
&\le \log\left(\sum_{i\in J}\frac{\ell_i(y)}{\ell_{m(y)}(y)}+1\right).
\end{align*}
The second inequality in Observation \ref{obsadj} implies that
\begin{align*}
\|f(x)-f(y)\|_p^p&\le 2\rho_T(x,y)^p\log(\sum_{i\in J}\frac{\ell_i(y)}{\ell_{m(y)}(y)}+1)+\rho_{T}(x,y)^p\left(\frac{2\kappa-1}{2\kappa}\right)^{p-1}\\
& \le 2\rho_T(x,y)^p\log(2\kappa|J|+1)+\rho_{T}(x,y)^p\left(\frac{2\kappa-1}{2\kappa}\right)^{p-1}\\
& \le (2\log(2\kappa^2+1)+1)\rho_{T}(x,y)^p\\
& \le 6\log(2\kappa)\rho_{T}(x,y)^p.
\end{align*}
\end{proof}
Claim \ref{lip}, together with Claim \ref{lipinv}, concludes the proof of Theorem \ref{catdistortion}.
\end{proof}
Let $\cdiam(T)$ denotes the combinatorial diameter of a weighted tree $T$, i.e. the diameter of $T$ for the metric induced by unit weights. The monotone coloring of every weighted tree that assign a different color to every edge being $\cdiam(T)$-caterpillar, Corollary \ref{cdiam} follows.

\begin{corollary}\label{cdiam} Let $p\in(1,\infty)$. For any weighted tree $T$ there exists a set $I$ and an embedding of $T$ into $\ell_p(I)$ with distortion $O_p(\log(\cdiam(T))^{\frac{1}{p}})$.
\end{corollary}

Corollary \ref{cdiam} may be known to the experts but we could not locate a proof in the literature.

\begin{rem}The caterpillar dimension of a finite tree can be estimated in polynomial time using dynamic programming, while estimating the combinatorial diameter can be done in linear time with an algorithm using a breadth-first search approach.
\end{rem}
Strong colorings were defined for finite trees in \cite{LeeNaorPeres2009}. The definition is readily extendable to arbitrary trees once the monotonicity of a path is defined for root-leafend paths and not only for root-leaf paths. A coloring $\chi\colon E\to\cC$ of a weighted tree $T=(V,E,w)$ is $\delta$-strong in the sense of \cite{LeeNaorPeres2009} if it is monotone, and for every $x,y\in V$, at least half of the shortest path connecting $x$ and $y$ is covered by color classes of length at least $\delta\rho_T(x,y)$, i.e.
$$\sum_{c\in \cC}\ell_c^\chi(x,y)\cdot \textbf{1}_{\{c\colon \ell_c^\chi(x,y)\ge \delta\rho_T(x,y)\}}\ge \frac{1}{2}\rho_T(x,y),$$
where $$\ell_c^\chi(x,y):=\sum_{\underset{e\in P(x,y)}{\chi(e)=c}}w(e).$$
Let $\delta^*(T):=\sup\{\delta \colon $T$ \textrm{ admits a $\delta$-strong coloring}\}$.
As already mentioned in \cite{LeeNaorPeres2009} for finite trees, a caterpillar coloring is a strong coloring.
\begin{lemma}\label{strongcaterpi} Let $T$ be a weighted tree and $\kappa\in\bn$. A $\kappa$-caterpillar coloring of $T$ is also a $\frac{1}{4\kappa}$-strong coloring. Therefore, $\delta^*(T)\ge\frac{1}{4\kappa_*(T)}.$
\end{lemma}
\begin{proof}
Let $\chi\colon E\to \cC$ be a $\kappa$-caterpillar coloring of $T=(V,E)$ and assume that $\chi$ is not a $\frac{1}{4\kappa}$-strong coloring. Then there exist $x,y\in V$ such that $$\sum_{c\in \cC}\ell_c^\chi(x,y)\cdot \textbf{1}_{\{c\colon \ell_c^\chi(x,y)\ge \frac{\rho_T(x,y)}{4\kappa}\}}<\frac{1}{2}\rho_T(x,y).$$
But,
\begin{align*}
\rho_T(x,y)&=\sum_{c\in \cC}\ell_c^\chi(x,y)\cdot \textbf{1}_{\{c\colon \ell_c^\chi(x,y)\ge \frac{\rho_T(x,y)}{4\kappa}\}}+\sum_{c\in \cC}\ell_c^\chi(x,y)\cdot \textbf{1}_{\{c\colon \ell_c^\chi(x,y)< \frac{\rho_T(x,y)}{4\kappa}\}}\\
& <\frac{1}{2}\rho_T(x,y)+2\kappa\frac{\rho_T(x,y)}{4\kappa}=\rho_T(x,y),
\end{align*}
a contradiction.
\end{proof}

However a strong coloring is not necessarily a caterpillar coloring. Indeed, consider a ray with countably many edges, and assign the sequence of weights $(\frac{1}{2}, \frac{1}{4}, \dots, \frac{1}{2^k},\dots)$ to the edges starting from the root. If every edge is colored with a different color, the monotone coloring obtained is not a $\kappa$-caterpillar coloring for any finite $\kappa$ but is a $\frac{1}{2}$-strong coloring. Proposition \ref{counterex} shows that the inequality in Lemma \ref{strongcaterpi} cannot be reversed in full generality.
\begin{proposition}\label{counterex}
There exists a weighted tree $T$ with $\delta^*(T)\ge\frac{1}{4}$ and $\kappa_*(T)=\infty$.
\end{proposition}

\begin{proof}Consider the combinatorial binary tree with infinite height $B_\infty$. It is clear that $\kappa_*(B_\infty)=\infty$. Put the weight $\frac{1}{2^n}$ on every edge $e=(x,y)$ whose vertices $x$ and $y$ are at distance respectively $n-1$ and $n$ to the root. The edge coloring where every edge has a different color is $\frac{1}{4}$-strong. Let $x,y\in B_\infty$. Let $lca(x,y)$ denotes the last common ancestor of $x$ and $y$, and $\ell(x)$ (resp. $\ell(y)$) the length of the edge attached to $lca(x,y)$ and pointing toward $x$ (resp. $y$). Note that every ray starting at the root is assigned the sequence of weights $(\frac{1}{2}, \frac{1}{4}, \dots, \frac{1}{2^n},\dots)$, and hence $\ell(x)=\ell(y)=\frac{1}{2^k}$ for some $k\in\bn$. Since for every $n\in\bn$, $\sum_{i=n}^\infty2^{-i}=2^{-k+1}$, one has $\rho_{B_\infty}(lca(x,y),x)<2^{-k+1}$, $\rho_{B_\infty}(lca(x,y),y)<2^{-k+1}$, and hence $\rho_{B_\infty}(x,y)<2^{-k+2}$. It follows that
\begin{align*}
\ell(x)+\ell(y)=2^{-k}+2^{-k}&>\frac{\rho_{B_\infty}(lca(x,y),x)}{2}+\frac{\rho_{B_\infty}(lca(x,y),y)}{2}=\frac{\rho_{B_\infty}(x,y)}{2},
\end{align*}
but $\min\{\ell(x);\ell(y)\}>\ds\frac{\rho_{B_\infty}(x,y)}{4}$.
\end{proof}

\section{Applications}\label{app}
\subsection{Stability of the asymptotic structure under nonlinear quotients}\label{stability}
Co-Lipschitz maps were introduced by Gromov in \cite{Gromov1998} in the context of geometric group theory. Later, Lipschitz quotients and uniform quotients were introduced and studied in the framework of Banach spaces by Bates, Johnson, Lindenstrauss, Preiss, and Schechtman \cite{BJLPS1999}. A map $f\colon X\to Y$ between metric spaces $X$ and $Y$ is called a uniform quotient map, and $Y$ is simply said to be a uniform quotient of $X$, if there exist non-decreasing functions $\rho,\omega\colon\br_+\to\br_+$ satisfying $\lim_{t\to 0}\omega(t)=0$ and $\rho(t)>0$ for all $t>0$ so that for all $x\in X$ and $r\in(0,\infty)$ one has
\begin{equation}\label{uniformquotient}
B_Y(f(x),\rho(r))\subset f(B_X(x,r))\subset B_Y(f(x),\omega(r)).
\end{equation}
If only the left inclusion in \eqref{uniformquotient} is satisfied then $f$ is said to be co-uniformly continuous. If the non-decreasing functions satisfy $\omega(r)\le Lr$ and $\rho(r)\ge r/C$ for some $L,C>0$, then $f$ is called a Lipschitz quotient map, and $Y$ is said to be a Lipschitz quotient of $X$. Note that the right inclusion in \eqref{uniformquotient} with $\omega(r)\le Lr$ is equivalent to $f$ being Lipschitz with $\lip(f)\le L$. If the left inclusion in \eqref{uniformquotient} is satisfied with $\rho(r)\ge r/C$, $f$ is said to be co-Lipschitz, and the infimum of all such $C$'s, denoted by $\colip(f)$, is called the co-Lipschitz constant of $f$. We define the \textit{codistortion} of a Lipschitz quotient map $f$ as $\codist(f):=\lip(f)\cdot\colip(f)$.

\smallskip

In the sequel mainly nonlinear quotient maps defined on some subset of a metric space are considered. We say that $Y$ is a Lipschitz (resp. uniform) subquotient of $X$ if $Y$ is a Lipschitz (resp. uniform) quotient of a subset of $X$. In particular a quantitative analysis of Lipschitz subquotients, similar to the quantitative theory of bi-Lipschitz embeddings, is emphasized.

\begin{defn}
Let $X,Y$ be two metric spaces. $Y$ is a said to be a Lipschitz subquotient of $X$ with codistortion $\alpha\in[1,\infty)$ (or simply $Y$ is an $\alpha$-Lipschitz subquotient of $X$) if there is a subset $Z\subset X$ and a Lipschitz quotient map $f\colon Z\to Y$ such that $\codist(f)\le\alpha$. We define the $X$-quotient codistortion of $Y$ as
$$qc_X(Y):=\inf\{\alpha\colon Y\textrm{ is an $\alpha$-Lipschitz subquotient of } X\}.$$
We set $qc_X(Y)=\infty$ if $Y$ is not a Lipschitz quotient of any subset of $X$.
\end{defn}

\begin{rem}
Lipschitz subquotients have already been implicitly touched upon (e.g. in \cite{MendelNaor2013}, \cite{LimaR2012}, \cite{DKR2014}). A ``dual'' notion was considered by Mendel and Naor in \cite{MendelNaor2004}, where given $\alpha\in[1,\infty)$ they say that $X$ has an $\alpha$-Lipschitz quotient in $Y$ if there is a subset $S\subset Y$ and a Lipschitz quotient map $f\colon X\to S$ such that $\codist(f)\le\alpha$.
\end{rem}

Observe that if $f$ is a bi-Lipschitz embedding from $X$ into $Y$, then $f^{-1}$ is a Lipschitz quotient map from $f(X)$ onto $X$, with $\codist(f^{-1})=\dist(f)$. Therefore we have $qc_Y(X)\le c_Y(X)$. A crucial and known observation for the ensuing discussion is that the previous inequality is actually an equality for trees.

\begin{proposition}\label{treesubquolip}
Let $Y$ be a metric space and $T$ a weighted tree. Then $qc_Y(T)=c_Y(T)$.
\end{proposition}

\begin{proof}
Let $Z$ be a subset of $Y$ and $f: Z\to T$ a Lipschitz quotient map. Equip $T$ with its canonical graph metric $\rho_T$ and root $T$ at an arbitrary vertex $r$ so that the height of the tree is well defined. By induction on the height of the tree it is possible to select a collection of points $(z_v)_{v\in T}\subset Z$ such that $f(z_v)=v$, and for every pair of adjacent vertices $(v,w)$ one has $d_Y(z_v,z_w)\le \colip(f)\rho_T(v,w)$. Since for a weighted graph it is sufficient to consider pairs of adjacent vertices to estimate the Lipschitz constant of a map, the injective map $g\colon v\mapsto z_v$ is Lipschitz with $\lip(g)\le\colip(f)$. We conclude by simply observing that $\lip(g^{-1})\le \lip(f)$, and hence $\dist(g)\le \codist(f)$.
\end{proof}

Recently, the stability under nonlinear quotients of the asymptotic structure of infinite-dimensional Banach spaces has been investigated (\cite{LimaR2012}, \cite{DKLR2014}, \cite{Zhang2015}, \cite{DKR2014}). The common feature of these articles is the implementation of a delicate and technical argument (or some slight variations of it) called ``fork argument'', which describes the behavior of a nonlinear lifting of points that are approximately in a fork configuration. This behavior depends heavily on the asymptotic geometry of the spaces and can rule out the existence of certain nonlinear quotient maps. As explained in \cite{LimaR2012}, the general idea is to built a collection of points approximately in a fork configuration in the target space whose set of pre-images contains a fork of comparable size, and then use the quantification of property ($\beta)$ to get a contradiction. Our work unifies, and extends, a series of results from \cite{LimaR2012}, \cite{DKLR2014}, \cite{Zhang2015}, and \cite{DKR2014}, which we now describe.

\medskip

The main motivation of Lima and Randrianarivony was to solve a long-standing open problem raised in \cite{BJLPS1999}. They proved that a Banach space that is a uniform quotient of $\ell_p$ for $1<p<2$ must be isomorphic to a linear quotient of $\ell_p$. The key ingredient was to show that $\ell_q$ cannot be a uniform quotient of $\ell_p$ if $1<p<q<\infty$. The authors already noticed that their proof will work equally well for Lipschitz subquotients. The following refinement of the Lima-Randrianarivony result appears in \cite{DKLR2014}.

\begin{theorem}[\cite{DKLR2014}]\label{refinement}
Let $X$ be a linear quotient of a subspace of an $\ell_p$-sum of finite-dimensional spaces, where $p\in(1,\infty)$. Assume that a Banach space $Y$ is a uniform subquotient of $X$, where the uniform quotient map is Lipschitz for large distances. Then $Y$ does not contain a subspace isomorphic to $\ell_q$ for any $q>p$.
\end{theorem}

Another implementation of the ``fork argument'' by Lima and Randrianarivony gives the following theorem.

\begin{theorem}[\cite{LimaR2012}]\label{couniform}
$\co$ is not a uniform quotient (or a Lipschitz subquotient) of a Banach space with property ($\beta)$.
\end{theorem}

Later, the second author of this article introduced in \cite{Zhang2015} the notion of coarse quotient map and proved Theorem \ref{lpcoarse} below, as well as a coarse analogue of Theorem \ref{couniform}, using a coarse version of the ``fork argument''.

\begin{theorem}[\cite{Zhang2015}]\label{lpcoarse}
Let $X$ be a Banach space with property ($\beta_p$) for some $p\in(1,\infty)$. Assume that a Banach space $Y$ is a coarse quotient of a subset of $X$, where the coarse quotient map is Lipschitz for large distances. Then $Y$ does not contain a subspace isomorphic to $\ell_q$ for any $q>p$.
\end{theorem}

In \cite{DKR2014}, Dilworth, Kutzarova, and Randrianarivony proved a nice rigidity result.

\begin{theorem}[\cite{DKR2014}]\label{rigidity}
If $Y$ is a separable Banach space that is a uniform quotient of a Banach space $X$ that has an equivalent norm with property ($\beta$), then $Y$ must be reflexive and admits an equivalent norm with property ($\beta$).
\end{theorem}
The fact that $Y$ must be reflexive is a consequence of one of the numerous James' characterizations of reflexivity. The core of their proof relies on Theorem \ref{treenosubquotient} below.

\begin{theorem}[\cite{DKR2014}]\label{treenosubquotient}
$T^{\omega}_{\omega}$ is not a Lipschitz subquotient of any Banach space admitting an equivalent norm with property ($\beta$).
\end{theorem}

Even though our approach is based on similar ideas, we circumvent the technical ``fork argument'' by splitting its proof mechanism into two distinct quantitative problems, interesting in their own right, that can be treated by rather elementary techniques. For instance, the proof in \cite{DKR2014} of Theorem \ref{treenosubquotient} is very clever but somehow delicate. The argument requires the introduction of the parasol graph with infinitely many levels and the implementation of the ``fork argument'' with respect to a certain type of liftings. Our alternative proof of Theorem \ref{treenosubquotient}, which is a direct consequence of Theorem \ref{treedist} and Proposition \ref{treesubquolip}, is elementary and avoids this highly technical and lengthy argument. It will also be clear in a moment that Theorem \ref{refinement}, Theorem \ref{couniform}, and Theorem \ref{lpcoarse} to a certain extent, fit naturally into the same framework.

\begin{defn}
Let $X$ and $Y$ be metric spaces and $I$ a subinterval of $[0,\infty)$. A map $f\colon X\to Y$ is called an $I$-range Lipschitz quotient map if there exist $C,L\in(0,\infty)$ depending on $I$ such that for all $r\in I$ and $x\in X$ one has
\begin{align}\label{Irange}
B_Y(f(x),\frac{r}{C})\subset f(B_X(x,r))\subset B_Y(f(x),Lr).
\end{align}
We say that $f$ is $I$-range Lipschitz (resp. $I$-range co-Lipschitz) if the right (resp. left) inclusion in \eqref{Irange} is satisfied for all $r\in I$ and $x\in X$.

In particular, we say that $f$ is Lipschitz (resp. co-Lipschitz) for large distances if it is $[s,\infty)$-range Lipschitz (resp. $[s,\infty)$-range co-Lipschitz) for every $s\in(0,\infty)$. $f$ is called a large scale Lipschitz quotient map, and we say that $Y$ is a large scale Lipschitz quotient of $X$ if $f$ is a $[s,\infty)$-range Lipschitz quotient map for every $s\in(0,\infty)$.
\end{defn}

Recall also that a metric space $X$ is said to be metrically convex if for every $x_0, x_1\in X$ and $t\in(0,1)$ there exists $x_t\in X$ so that $d_X(x_0, x_t)=td_X(x_0, x_1)$ and $d_X(x_1, x_t)=(1-t)d_X(x_0, x_1)$. It is a classical fact that a uniform quotient map between metrically convex spaces is actually a large scale Lipschitz quotient map (cf. \cite{BLbook} Proposition 1.11 and Lemma 11.11). The quantitative stand that we have taken turns out to be extremely efficient due to Proposition \ref{discrepancy} below, a qualitative version of which (in the special case where $G_n=T_\omega^\omega$) is implicit in \cite{DKR2014}.

\begin{proposition}\label{discrepancy}
Let $X$ and $Y$ be Banach spaces such that $Y$ is a uniform subquotient of $X$, where the uniform quotient map is Lipschitz for large distances. Let $(G_n)_{n=1}^\infty$ be a sequence of unweighted connected simple graphs. Then $qc_X(G_n)=O(qc_Y(G_n))$ for all $n\in\bn$.
\end{proposition}

\begin{proof}
Let $Z$ be a subset of $X$ and let $f:Z\rightarrow Y$ be a uniform quotient map that is Lipschitz for large distances. Assume that for every $n\ge 1$, $S_n$ is a subset of $Y$ and $g_n: S_n\rightarrow G_n$ is a Lipschitz quotient map. By a scaling of the set $S_n$ we may without loss of generality assume that $\text{Lip}(g_n)=1$. Let $f_n$ denote the restriction of $f$ to $Z_n:=f^{-1}(S_n)$ and $\lip_{t}(f):=\sup\{\|f(x)-f(y)\|_Y: \|x-y\|_X\ge t\}$ the Lipschitz constant of $f$ for distances larger than $t$. Then for every $n\ge 1$ and every $t\in(0,\infty)$,  $\lip_{t}(f_n)\le\lip_{t}(f)<\infty$. Next it is shown that the maps $h_n:=g_n\circ f_n$ are Lipschitz quotient maps from $Z_n$ onto $G_n$.

\begin{claim}\label{liphn} There exists $\delta\in(0,\infty)$ such that $\lip(h_n)\le\lip_{\delta}(f)$ for all $n\in\bn$.
\end{claim}
\begin{proof}[Proof of Claim \ref{liphn}]\renewcommand{\qedsymbol}{}
Since $f$ is uniformly continuous, there exists $\delta\in(0,\infty)$ so that $\|f(x)-f(y)\|_Y<1$ whenever $\|x-y\|_X<\delta$. For every $x, y\in Z_n$ such that $\|x-y\|<\delta$ one has $h_n(x)=h_n(y)$ since $$\rho_{G_n}(h_n(x),h_n(y))\le\|f_n(x)-f_n(y)\|<1.$$
If $\|x-y\|\ge\delta$ then $$\rho_{G_n}(h_n(x),h_n(y))\le\|f_n(x)-f_n(y)\|\le\lip_{\delta}(f)\|x-y\|.$$
\vskip -.5cm
\end{proof}

\begin{claim}\label{coliphn} There exists $c\in(0,\infty)$ such that $\colip(h_n)\le (c+1)\colip(g_n)$ for all $n\in\bn$.
\end{claim}
\begin{proof}[Proof of Claim \ref{coliphn}]\renewcommand{\qedsymbol}{}
Denote $\text{coLip}(g_n):=D_n\in[1,\infty)$. Since $Y$ is a Banach space it is metrically convex, and hence $f$ as well as its restrictions to $Z_n$ are co-Lipschitz for large distances, there exists $c\in(0,\infty)$ such that for all $n\ge 1$, for all $x\in Z_n$, and for all $r\ge1$,
$$B_{S_n}(f_n(x),\frac{r}{c})\subset f_n(B_{Z_n}(x,r)).$$
For every $x\in Z_n$ one has
\begin{align*}
B_{G_n}(h_n(x),1)&\subset g_n(B_{S_n}(f_n(x),D_n))\\
&\subset g_n\left(B_{S_n}\left(f_n(x),\frac{(c+1)D_n}{c}\right)\right)\subset h_n(B_{Z_n}(x,(c+1)D_n)).
\end{align*}
It follows that $\text{coLip}(h_n)\le (c+1)\colip(g_n)$ since the $G_n$'s are connected graphs. Indeed, to show that the maps $h_n$ are co-Lipschitz with constant, say $C$, it is sufficient to show that $B_{G_n}(h_n(x),1)\subset h_n(B_{Z_n}(x,C))$.
\end{proof}
Therefore there exist $\delta,c\in(0,\infty)$ so that $\codist(h_n)\le (c+1)\lip_{\delta}(f)\codist(g_n)$ for every $n\in\bn$.
\end{proof}

In regards of Proposition \ref{discrepancy} and Proposition \ref{treesubquolip}, our alternative proofs of (stronger forms) of Theorem \ref{refinement} and Theorem \ref{couniform} simply boil down to exhibiting a discrepancy between the $Y$-distortion and the $X$-distortion of the complete countably branching trees. This discrepancy is exhibited by comparing the lower bound from Theorem \ref{treedist} with the upper bound from Corollary \ref{cdiam}.

\begin{theorem}\label{betaqstrong}
Let $X$ be a Banach space admitting an equivalent norm with property ($\beta_p$) for some $p\in(1,\infty)$. Assume that a Banach space $Y$ is a uniform subquotient of $X$, where the uniform quotient map is Lipschitz for large distances. Then $\ell_{q}$ is not a uniform subquotient of $Y$ for any $q>p$ such that the uniform quotient map is Lipschitz for large distances.
\end{theorem}

\begin{proof}
If $X$ is a Banach space with property ($\beta_p$) for some $p\in(1,\infty)$, then $c_X(T^\omega_h)=\Omega(\log(h)^{\frac{1}{p}})$. Now, if $\ell_{q}$ is a uniform subquotient of $Y$ for some $q>p$ such that the uniform quotient map is Lipschitz for large distances and if $Y$ is a uniform subquotient of $X$, where the uniform quotient map is also Lipschitz for large distances, then it follows from Proposition \ref{treesubquolip}, Proposition \ref{discrepancy} and Corollary \ref{cdiam} that $c_X(T^\omega_h)=O(\log(h)^{\frac{1}{q}})$. There is a contradiction for $h$ big enough.
\end{proof}

\begin{theorem}\label{costrong}
$\co$ is not a uniform subquotient of a Banach space admitting an equivalent norm with property ($\beta$) such that the uniform quotient map is Lipschitz for large distances.
\end{theorem}
\begin{proof}
Assume that $\co$ is a uniform subquotient of a Banach space $X$ admitting an equivalent norm with property ($\beta$) such that the uniform quotient map is Lipschitz for large distances. Then it follows from Proposition \ref{discrepancy} that $qc_X(T_h^\omega)=O(qc_{\co}(T_h^\omega))$ for all $h\in\bn$, but it is easy to show using the summing basis that $c_{\co}(T_h^\omega)\le 2$ (actually that $c_{\co}(T_h^\omega)=1$ follows from Theorem 6.3 in \cite{KaltonLancien2008}). Since $qc_X(T^\omega_h)=c_X(T^\omega_h)$  and $qc_{\co}(T^\omega_h)=c_{\co}(T^\omega_h)$ by Proposition \ref{treesubquolip}, one has $c_X(T_h^\omega)=O(c_{\co}(T_h^\omega))=O(1)$ for all $h\in\bn$, but this contradicts Theorem \ref{treedist}.
\end{proof}

\begin{rem} The conclusion of Theorem \ref{betaqstrong} (resp. Theorem \ref{costrong}) can be strengthened. Indeed, only the fact that $Y$ satisfies $c_Y(T_h^\omega)=o(\log(h)^{\frac{1}{p}})$ is needed (resp. $\co$ can be replaced by any Banach space $Y$ such that $c_Y(T_h^\omega)=o(\log(h)^{\frac{1}{p}})$ for every $p\in(1,\infty)$).
\end{rem}

The case of coarse quotients is a bit more delicate. A map $f\colon X\to Y$ between two metric spaces $X$ and $Y$ is said to be coarsely continuous if $\omega_f(t)<\infty$ for all $t>0$, where $\omega_f$ is the expansion modulus of $f$ defined by
$$\omega_f(t):=\sup\{d_Y(f(x),f(y)): d_X(x,y)\le t\}.$$
$f$ is said to be co-coarsely continuous with constant $K\in[0,\infty)$ if for every $\varepsilon>0$ there exists $\delta:=\delta(\varepsilon)>0$ so that for every $x\in X$, $$B_Y(f(x),\varepsilon)\subset f(B_X(x,\delta))^K,$$
where for a subset $Z$ of a metric space $Y$ the notation $Z^K$ means the $K$-neighborhood of $Z$, i.e., $Z^K:=\{y\in Y\colon d_Y(y,z)\le K~\text{for some}~z\in Z\}$. A map $f$ is then said to be a coarse quotient map if $f$ is both co-coarsely continuous and coarsely continuous, and in that case we say $Y$ is a coarse quotient of $X$. $Y$ is said to be a coarse subquotient of $X$ if $Y$ is a coarse quotient of a subset of $X$.

\smallskip

The technical lemma below is needed to prove an analogue of Proposition \ref{discrepancy} in the coarse setting. The proof can be found in \cite{Zhang2015} in a slightly different context. Roughly speaking it says that a subset of a quotient is actually a quotient of a subset. Note that this argument is straightforward in the uniform case, but in the coarse setting it requires some effort. For the sake of completeness the proof is presented here.

\begin{lemma}\label{csubquo}
Let $X$ and $Y$ be metric spaces and $f:X\rightarrow Y$ a coarse quotient map with constant $K$. Assume that $Y$ is metrically convex and $S$ is a subset of $Y$. Then there exist a subset $Z\subset X$ and a map $g:Z\rightarrow S$ satisfying the following:
\begin{enumerate}[(i)]
\item\label{case1} If $K=0$, then for every $\varepsilon>0$ there exists $c_1:=c_1(\varepsilon)>0$ such that for all $x\in Z$ and $r\ge\varepsilon$,
    \begin{align}\label{Kzero}
    B_S(g(x),r)\subset g(B_Z(x,c_1r)).
    \end{align}
\item\label{case2} If $K>0$, then there exists $c_2:=c_2(K)>0$ such that for all $x\in Z$ and $r>0$,
    \begin{align}\label{Kpositive}
    B_S(g(x),r)\subset g(B_Z(x,c_2r))^{4K}.
    \end{align}
\end{enumerate}
\end{lemma}

\begin{proof}
First we claim that $f$ satisfies the following property:

\smallskip

\noindent For every $\varepsilon>2K$, there exists $c:=c(\varepsilon)>0$ so that for all $x\in X$ and $r\ge\varepsilon$, $$B_Y(f(x),r)\subset f(B_X(x,cr))^K.$$

\noindent Indeed, let $n$ be the positive integer so that $(n-1)\varepsilon\le r<n\varepsilon$ and assume that $y\in B_Y(f(x),r)$. Since $Y$ is metrically convex, there exist $\{y_i\}_{i=0}^{2n}$ with $y_0=f(x)$ and $y_{2n}=y$ such that $d(y_{i},y_{i-1})\le\frac{\varepsilon}{2}$ for all $i$. It follows from the definition of co-coarse continuity that there exists $\delta:=\delta(\varepsilon)>0$ such that $y_1\in B_Y(f(x),\varepsilon)\subset f(B_X(x,\delta))^K$, so $d_Y(y_1,f(x_1))\le K$ for some $x_1\in B_X(x,\delta)$, and hence it follows from the triangle inequality that $y_2\in B_Y(f(x_1),\varepsilon)$. We proceed inductively to get $\{x_i\}_{i=1}^{2n}$ such that $d_X(x_{i},x_{i-1})\le\delta$ and $d_Y(y_i,f(x_i))\le K$ for all $i$. This implies $y\in f(B_X(x,2n\delta))^K\subset f(B_X(x,cr))^K$, where $c:=c(\varepsilon)=\frac{4\delta(\varepsilon)}{\varepsilon}$.

\medskip

Define $p: S^K\rightarrow S$ by $p(a)=a$ if $a\in S$ and $p(a)=s_a$ otherwise, where $s_a$ is any point in $S$ within distance $K$ from $a$. We now show that in both cases \eqref{case1} and \eqref{case2} one can take $Z=f^{-1}(S^K)$ and $g=p\circ\widetilde{f}$, where $\widetilde{f}: Z\to S^K$ is the restriction of $f$ to $Z$.

Indeed, in case \eqref{case1} when $K=0$, the map $p$ becomes the identity map on $S$ and hence $g: Z\to S$ is the restriction of $f$ to $Z=f^{-1}(S)$. Thus \eqref{Kzero} follows with $c_1(\varepsilon)=c(\varepsilon)$ by the above claim.

In case \eqref{case2} when $K>0$, first observe that the claim still holds for $\varepsilon=2K$, i.e. there exists $\tilde{c}=c(2K)>0$ so that for all $x\in X$ and $r\ge2K$ one has $B_Y(f(x),r)\subset f(B_X(x,\tilde{c}r))^K.$ Now for $x\in Z$ and $r\ge2K$, suppose that $y\in B_{S^K}(\widetilde{f}(x),r)$. Then there exists $u\in B_X(x,\tilde{c}r)$ such that $d_Y(y, f(u))\le K$, and $y\in S^K$ implies that $d_Y(y,s)\leq K$ for some $s\in S$, so
$$s\in B_Y(f(u),2K)\subset f(B_X(u,2K\tilde{c}))^K.$$
Thus there exists $v\in B_X(u,2K\tilde{c})$ such that $d_Y(s,f(v))\le K$ and hence $v\in Z$. It follows that $d_Y(y,\widetilde{f}(v))\le2K$ and $v\in B_Z(x,2K\tilde{c}+\tilde{c}r)\subset B_Z(x,2\tilde{c}r)$, so we have shown that the map $\widetilde{f}: Z\to S^K$ satisfies
$$B_{S^K}(\widetilde{f}(x),r)\subset \widetilde{f}(B_Z(x,2\tilde{c}r))^{2K}$$
for all $x\in Z$ and $r\ge2K$. Therefore for every $x\in Z$ and $r\ge4K$ we have
\begin{align*}
B_S(g(x),r)&\subset p(B_{S^K}(\widetilde{f}(x),r+K))\subset p(B_{S^K}(\widetilde{f}(x),2r))\\
&\subset p\left(\widetilde{f}(B_Z(x,4\tilde{c}r))^{2K}\right)\subset \left(p\circ\widetilde{f}(B_Z(x,4\tilde{c}r))\right)^{\omega_p(2K)}\\
&\subset g(B_Z(x,4\tilde{c}r))^{4K}.
\end{align*}
This implies that \eqref{Kpositive} holds for $c_2:=c_2(K)=4\tilde{c}$.
\end{proof}

\begin{rem}
The map $g$ is actually a coarse quotient map with constant $4K$ even if $Y$ is not metrically convex.
\end{rem}
The next proposition is the analogue of Proposition \ref{discrepancy} that is needed in the coarse case.
\begin{proposition}\label{discrepancycoarse}
Let $X$ and $Y$ be Banach spaces such that $Y$ is a coarse subquotient of $X$, where the coarse quotient map is Lipschitz for large distances. Then there exists $k\in\bn$ (independent of $n$) so that $qc_X(T^\omega_{2^n})=O(qc_Y(T^\omega_{2^{n+k}}))$ for all $n\in\bn$.
\end{proposition}

\begin{proof}
Let $Z$ be a subset of $X$ and let $f:Z\rightarrow Y$ be a coarse quotient map with constant K that is Lipschitz for large distances, i.e. $\lip_t(f)<\infty$ for all $t\in(0,\infty)$. We claim that $k$ can be chosen as the smallest positive integer so that $2^k>\omega_f(1)+4K+1$. Assume that $S_n$ is a subset of $Y$ and $g_n: S_n\rightarrow T_{2^{n+k}}^{\omega}$ is a Lipschitz quotient map. By a scaling of the set $S_n$ we may without loss of generality assume that $\text{Lip}(g_n)=1$. There exist a subset $T(n)\subset T_{2^{n+k}}^{\omega}$, whose distance between points in $T(n)$ is at least $2^k$ and a rescaled isometry $i_n\colon T(n)\to T_{2^n}^{\omega}$ so that $\rho_{{T_{2^n}^\omega}}(i_n(u),i_n(v))=2^{-k}\rho_{T_{2^{n+k}}^{\omega}}(u,v)$ for every $u,v\in T(n)$. Let $\widetilde{S}_n:=g_n^{-1}(T(n))$. By Lemma \ref{csubquo} there exists $c>0$ depending only on $K$, so that for every $n\in\bn$ there exist sets $Z_n\subset Z$ and coarse quotient maps $f_n:Z_n\rightarrow\widetilde{S}_n$ satisfying for all $x\in Z_n$ and $r\ge4K+1$,
$$B_{\widetilde{S}_n}(f_n(x),r)\subset f_n(B_{Z_n}(x,cr))^{4K}.$$
The following diagram summarizes the situation:
$$\begin{array}{ccccccc}
X      & & & & & & \\
 \cup & & & & & & \\
Z      &\overset{f}{\relbar\joinrel\relbar\joinrel\relbar\joinrel\relbar\joinrel\relbar\joinrel\relbar\joinrel\rightarrow} & Y     & & & & \\
        &                                            & \cup & & & & \\
        &                                            & S_n & \overset{g_n}{\relbar\joinrel\relbar\joinrel\relbar\joinrel\relbar\joinrel\relbar\joinrel\relbar\joinrel\rightarrow} & T_{2^{n+k}}^{\omega}& & \\
\cup &                                            & \cup &                                                 & \cup                            & & \\
 Z_n &\overset{f_n}{\relbar\joinrel\relbar\joinrel\relbar\joinrel\relbar\joinrel\relbar\joinrel\relbar\joinrel\rightarrow}& \widetilde{S}_n&\overset{\widetilde{g}_n:={g_n}_{|\widetilde{S}_n}}{\relbar\joinrel\relbar\joinrel\relbar\joinrel\relbar\joinrel\relbar\joinrel\relbar\joinrel\rightarrow}& T(n) &\overset{i_n}{\relbar\joinrel\relbar\joinrel\relbar\joinrel\relbar\joinrel\relbar\joinrel\relbar\joinrel\rightarrow} & T_{2^n}^{\omega}.\\
\end{array}$$

Consider the map $h_n:=i_n\circ\widetilde{g}_n\circ f_n: Z_n\to T_{2^n}^\omega$, where $\widetilde{g}_n$ is the restriction of $g_n$ to $\widetilde{S}_n$.

\begin{claim}\label{liphncoarse} For every $n\in\bn$, $\lip(h_n)\le2^{-k}(2K+\lip_1(f))$.
 \end{claim}
 \begin{proof}[Proof of Claim \ref{liphncoarse}]\renewcommand{\qedsymbol}{}
 For every $x, y\in Z_n$ such that  $\|x-y\|<1$ one has $h_n(x)=h_n(y)$ since $$\rho_{{T_{2^n}^\omega}}(h_n(x),h_n(y))\le2^{-k}\|f_n(x)-f_n(y)\|\le2^{-k}(2K+\|f(x)-f(y)\|)<1.$$
 If $\|x-y\|\ge1$ then $$\rho_{{T_{2^n}^\omega}}(h_n(x),h_n(y))\le2^{-k}(2K+\|f(x)-f(y)\|)\le2^{-k}(2K+\lip_1(f))\|x-y\|.$$ Therefore $\text{Lip}(h_n)\le2^{-k}(2K+\lip_1(f))$.
 \end{proof}
 \begin{claim}\label{coliphncoarse} For every $n\in\bn$, $\colip(h_n)\le2^kc\cdot\colip(g_n)$.
 \end{claim}
 \begin{proof}[Proof of Claim \ref{coliphncoarse}]\renewcommand{\qedsymbol}{} Denote $\text{coLip}(g_n):=D_n\in[1,\infty)$. For every $x\in Z_n$ one has
\begin{align*}
B_{T_{2^n}^\omega}(h_n(x),1)=i_n\left(B_{T(n)}(\widetilde{g}_n\circ f_n(x),2^k)\right)&\subset i_n\circ\widetilde{g}_n\left(B_{\widetilde{S}_n}(f_n(x),2^kD_n)\right)\\
 \subset i_n\circ\widetilde{g}_n\left(f_n\left(B_{Z_n}(x,2^kD_nc)\right)^{4K}\right)&\subset i_n\left(\left(\widetilde{g}_n\circ f_n\left(B_{Z_n}(x,2^kD_nc)\right)\right)^{4K}\right)\\
 &=h_n(B_{Z_n}(x,2^kD_nc)),
\end{align*}
which implies that $\colip(h_n)\le2^kD_nc$.
\end{proof}
\noindent Thus $h_n$ is a Lipschitz quotient map from $Z_n$ onto $T_{2^n}^\omega$, with $\codist(h_n)\le c(2K+\lip_1(f))\codist(g_n),$ where the constant $c(2K+\lip_1(f))$ depends only on $f$ and $K$.

\end{proof}

Now a combination of Proposition \ref{discrepancycoarse}, Proposition \ref{treesubquolip}, Corollary \ref{cdiam}, and Theorem \ref{treedist} gives:

\begin{theorem}
Let $X$ be a Banach space admitting an equivalent norm with property ($\beta_p$) for some $p\in(1,\infty)$. Assume that a Banach space $Y$ is a coarse subquotient of $X$, where the coarse quotient map is Lipschitz for large distances. Then $\ell_{q}$ is not a coarse subquotient of $Y$ for any $q>p$ such that the coarse quotient map is Lipschitz for large distances.
\end{theorem}

\begin{theorem}
$\co$ is not be a coarse subquotient of a Banach space admitting an equivalent norm with property ($\beta$) so that the coarse quotient map is Lipschitz for large distances.
\end{theorem}

\subsection{Metric characterization of asymptotic properties}
It is a celebrated result of Bourgain \cite{Bourgain1986a} that superreflexivity can be characterized in terms of the bi-Lipschitz embeddability of the complete hyperbolic binary trees. Since then other characterizations have been discovered \cite{Baudier2007}, \cite{JohnsonSchechtman2009}, \cite{Ostrovskii2014AGMS}. The asymptotic analogue of Bourgain's characterization was proved by the first author, Kalton, and Lancien \cite{BKL2010}. The definitions of the asymptotic versions of uniform convexity and uniform smoothness are briefly recalled. Let $(X,\|\cdot\|)$ be a Banach space and $t>0$. We denote by $S_X$ its unit sphere. For $x\in S_X$ and $Y$ a closed linear subspace of $X$, we define $$\overline{\rho}(t,x,Y):=\sup_{y\in S_Y}\|x+t y\|-1\ \ \ \ {\rm and}\ \ \ \ \overline{\delta}(t,x,Y):=\inf_{y\in S_Y}\|x+t y\|-1,$$ and
$$\overline{\rho}(t):=\sup_{x\in S_X}\ \inf_{{\rm dim}(X/Y)<\infty}\overline{\rho}(t,x,Y)\ \ \ \ {\rm and}\ \ \ \ \overline{\delta}(t):=\inf_{x\in S_X}\ \sup_{{\rm dim}(X/Y)<\infty}\overline{\delta}(t,x,Y).$$ The norm $\|\cdot \|$ is said to be {\it asymptotically uniformly smooth} (a.u.s. in short) if $$\lim_{t \to 0}\frac{\overline{\rho}(t)}{t}=0.$$ It is said to be {\it asymptotically uniformly convex} (a.u.c. in short) if $$\forall t>0\ \ \ \ \overline{\delta}(t)>0.$$ These moduli were introduced by Milman in \cite{Milman1971}. We recall the main result from \cite{BKL2010}.

\begin{theorem}[\cite{BKL2010}]\label{BKL} Let $X$ be a reflexive Banach space. The following assertions are equivalent:
\begin{enumerate}[(i)]
\item $X$ is a.u.s. renormable \underline{and} $X$ is a.u.c. renormable,
\item $\sup_{h\ge 1}c_X(T^\omega_h)=\infty$,
\item $c_X(T^\omega_\omega)=\infty$.
\end{enumerate}
\end{theorem}

Theorem \ref{equivalences}, which is partially explicit in \cite{DKLR2014}, follows from the proof of Theorem 4 in \cite{Kutzarova1990}.
\begin{theorem}[\cite{DKLR2014},\cite{Kutzarova1990}]\label{equivalences} Let X be a separable Banach space. The following assertions are equivalent:
\begin{enumerate}[(i)]
\item $X$ admits an equivalent norm with property $(\beta)$,
\item $X$ admits an equivalent norm with property $(\beta_p)$ for some $p\in(1,\infty)$,
\item $X$ is reflexive, a.u.s. renormable, and a.u.c. renormable.
\end{enumerate}
\end{theorem}

The same equivalences also hold without the separability assumption \cite{DKLRpriv}. If one uses a combination of Theorem \ref{BKL} and Theorem \ref{equivalences} to prove that, for every Banach space $Y$ admitting an equivalent norm with property ($\beta$) one has $\lim_{h\to\infty}c_Y(T_h^\omega)=\infty$, then one does not obtain an optimal estimate on the rate of growth of $(c_Y(T_h^\omega))_{h\ge 1}$. Moreover the proof in \cite{BKL2010} of the fact that $\lim_{h\to\infty}c_Y(T_h^\omega)=\infty$ for every reflexive Banach space $Y$ that is a.u.s. renormable and a.u.c. renormable is rather technical and escapes geometric intuition. The advantage of using Theorem \ref{treedist} stems for the fact that it gives a simple, geometric, and direct proof of the former fact, and it provides an optimal estimate on the rate of growth. New problems are also uncovered. Indeed, we showed that $\lim_{\ell\to\infty}c_Y(P_{\ell}^\omega)=\infty$, but the following related embedding problem is open.

\begin{problem}
If $Y$ does not admit any equivalent norm with property $(\beta)$, do we have $\sup_{\ell\ge 1}c_Y(P^\omega_\ell)<\infty$?
\end{problem}

\subsection{Finite determinacy of bi-Lipschitz embeddability problems}\label{determinacy}

Let $\lambda\in[1,\infty)$. A metric space $X$ is $\lambda$-finitely representable in another metric space $Y$ if $c_{Y}(F)\le \lambda$ for every \textit{finite} subset $F$ of $X$. $X$ is crudely finitely representable (resp. finitely representable) in $Y$ if it is $\lambda$-finitely representable in $Y$ for some $\lambda\in[1,\infty)$ (resp. for every $\lambda\in(1,\infty)$).

\medskip

Let $\cC$ be a class of metric spaces. Given a metric space $X$, we say that its bi-Lipschitz embeddability problem in the class $\cC$ is finitely determined if for \textit{every} $Y\in \cC$, $X$ admits a bi-Lipschitz embedding into $Y$ whenever $X$ is crudely finitely representable in $Y$. Ostrovskii's finite determinacy theorem \cite{Ostrovskii2012} says that for every locally finite metric space $X$, its bi-Lipschitz embeddability problem in the class of Banach spaces is finitely determined. It is folklore that the local finiteness condition in Ostrovskii's theorem cannot be removed. For instance, $\ell_2$ is finitely representable in $\ell_1$, but it is a well-known fact in nonlinear Banach space theory that $\ell_2$ does not bi-Lipschitzly embed into $\ell_1$. If one restricts ones attention to the class of graph metrics it becomes a tricky task to find examples of non-locally finite graphs whose bi-Lipschitz embeddability problem in the class of Banach spaces is \textit{not} finitely determined. Such an example can be provided appealing to Theorem \ref{treedist}. Indeed, if $Y=\xbl$, then $T_\omega^\omega$ is finitely representable in $Y$ but it does not admit any bi-Lipschitz embedding into $Y$. Therefore Ostrovskii's finite determinacy theorem does not hold even for structurally simple graphs such as (non-locally finite) trees. Ostrovskii's proof actually gives a more precise quantitative statement.

\begin{theorem}[\cite{Ostrovskii2012}] There exists $\mu\in(0,\infty)$ such that for every locally finite metric space $M$ and every Banach space $Y$ the inequality $c_Y(M)\le \mu\lambda$ holds whenever $M$ is $\lambda$-finitely representable in $Y$.
\end{theorem}

The bi-Lipschitz embeddability problem for the space $T_h^\omega$ seems to be more elusive. In the next proposition it is shown that an analogue of the quantitative statement above does not hold for the sequence $(T^\omega_h)_{h\ge 1}$.

\begin{proposition}\label{uniformfindet}Let $p\in(1,2)$. There does not exist a constant $\mu\in(0,\infty)$ such that for every $h\ge 1$ the inequality $c_{\ell_p}(T^\omega_h)\le \mu\lambda$ holds whenever $T^\omega_h$ is $\lambda$-finitely representable in $\ell_p$.
\end{proposition}

\begin{proof} Assume that there exists a finite constant $\mu>0$ such that for every $h\in\bn$, one has $c_{\ell_p}(T^\omega_h)\le \mu\lambda$ whenever $T_h^\omega$ is $\lambda$-finitely representable in $\ell_p$. Let $F$ be a finite subset of $T^\omega_h$. It follows from Corollary \ref{cdiam} that $c_{\ell_2}(F)=O(\sqrt{\log(h)})$ and hence $c_{\ell_p}(F)=O(\sqrt{\log(h)})$ by Dvoretzky's theorem. But Theorem \ref{treedist} says that $c_{\ell_p}(T^\omega_h)=\Omega( \log(h)^{\frac{1}{p}})$. Therefore $\mu\gtrsim\log(h)^{\frac{1}{p}-\frac{1}{2}}$ which is a contradiction when $h$ is large enough.
\end{proof}

\bigskip

\textbf{Acknowledgments:} We wish to thank G. Godefroy, W. B. Johnson, and G. Lancien for many inspirational discussions at various stages of the development of this article. We would like also to extend our gratitude and appreciation to F. Lancien, G. Lancien, and A. Proch\'azka for the flawless organization of the Autumn School on Nonlinear Geometry of Banach Spaces and Applications in M\'etabief, and of the Conference on Geometric Functional Analysis and its Applications in Besan\c con. The scientific activity and atmosphere was incredibly enlightening. The work presented here found its inspiration, and was initiated while participating at these events.


\begin{bibsection}
\begin{biblist}

\bib{AyerbeDominguezCutillas1994}{article}{
    AUTHOR = {Ayerbe, J. M.},
    AUTHOR = {Dom{\'{\i}}nguez Benavides, T.},
    AUTHOR = {Cutillas, S. F.},
     TITLE = {Some noncompact convexity moduli for the property {$(\beta)$}
              of {R}olewicz},
   JOURNAL = {Comm. Appl. Nonlinear Anal.},
  FJOURNAL = {Communications on Applied Nonlinear Analysis},
    VOLUME = {1},
      YEAR = {1994},
    NUMBER = {1},
     PAGES = {87--98},
      ISSN = {1074-133X},
   MRCLASS = {46B20 (47H09)},
  MRNUMBER = {1268081 (95a:46021)},
MRREVIEWER = {Elisabetta Maluta},
}

\bib{BJLPS1999}{article}{
  author={Bates, S.},
  author={Johnson, W. B.},
  author={Lindenstrauss, L.},
  author={Preiss, D.},
  author={Schechtman, G.},
  title={Affine approximation of Lipschitz functions and nonlinear quotients},
  journal={Geom. Funct. Anal.},
  volume={9},
  date={1999},
  pages={1092--1127},
}

\bib{Baudier2007}{article}{
  author={Baudier, F.},
  title={Metrical characterization of super-reflexivity and linear type of Banach spaces},
  journal={Archiv Math.},
  volume={89},
  date={2007},
  pages={419--429},
}


\bib{BKL2010}{article}{
  author={Baudier, F.},
  author={Kalton, N. J.},
  author={Lancien, G.},
  title={A new metric invariant for Banach spaces},
  journal={Studia Math.},
  volume={199},
  date={2010},
  pages={73--94},
}

\bib{BaudierLancien2008}{article}{
  author={Baudier, F.},
  author={Lancien, G.},
  title={Embeddings of locally finite metric spaces into Banach spaces},
  journal={Proc. Amer. Math. Soc.},
  volume={136},
  date={2008},
  pages={1029--1033},
}

\bib{BLbook}{book}{
  title={Geometric Nonlinear Functional Analysis. Vol. 1},
  author={Benyamini, Y.},
  author={Lindenstrauss, J.},
  series={American Mathematical Society Colloquium Publications},
  volume={48},
  place={Providence, RI},
  year={2000},
  publisher={American Mathematical Society}
}

\bib{Bourgain1986a}{article}{
  author={Bourgain, J.},
  title={The metrical interpretation of superreflexivity in Banach spaces},
  journal={Israel J. Math.},
  volume={56},
  date={1986},
  pages={222--230},
}

\bib{DKLRpriv}{article}{
  author={Dilworth, S. J.},
  author={Kutzarova, D.},
  author={Lancien, G.},
  author={Randrianarivony, N. L.},
  title={private communication},
}

\bib{DKLR2014}{article}{
  author={Dilworth, S. J.},
  author={Kutzarova, D.},
  author={Lancien, G.},
  author={Randrianarivony, N. L.},
  title={Asymptotic geometry of Banach spaces and uniform quotient maps},
  journal={Proc. Amer. Math. Soc.},
  fjournal={Proceedings of the American Mathematical Society},
  volume={142},
  year={2014},
  number={8},
  pages={2747--2762},
  issn={0002-9939},
  mrclass={46B80},
  mrnumber={3209329},
  url={http://dx.doi.org/10.1090/S0002-9939-2014-12001-6},
}

\bib{DKR2014}{article}{
  author={Dilworth, S. J.},
  author={Kutzarova, D.},
  author={Randrianarivony, N. L.},
  title={The transfer of property ($\beta $) of {R}olewicz by a uniform quotient map},
  journal={Trans. Amer. Math. Soc.},
  volume={},
  date={},
  pages={to appear, arXiv:1408.6424 (2014), 18 pages},
}


\bib{Gromov1998}{book}{
  title={Metric Structures for Riemannian and Non-Riemannian Spaces},
  series={Progress in Math.},
  volume={152},
  author={Gromov, M.},
  publisher={Birkh\"{a}user},
  place={Boston},
  date={1998},
}

\bib{Gupta2000}{article}{
  author={Gupta, A.},
  title={Embedding tree metrics into low-dimensional {E}uclidean spaces},
  journal={Discrete Comput. Geom.},
  fjournal={Discrete \& Computational Geometry. An International Journal of Mathematics and Computer Science},
  volume={24},
  year={2000},
  number={1},
  pages={105--116},
  issn={0179-5376},
  coden={DCGEER},
  mrclass={68U05 (54C25 54E35)},
  mrnumber={1765236 (2001b:68144)},
  url={http://dx.doi.org/10.1145/301250.301434},
}

\bib{GuKL2003}{article}{
 AUTHOR = {Gupta, A.},
 AUTHOR = {Krauthgamer, R.},
 AUTHOR = {Lee, J. R.},
     TITLE = {Bounded geometries, fractals, and low-distortion embeddings},
   JOURNAL = {in ``44th Symposium on Foundations of Computer Science''},
  FJOURNAL = {Discrete \& Computational Geometry. An International Journal
              of Mathematics and Computer Science},
      YEAR = {2003},
     PAGES = {534--543},

}

\bib{Handbook}{collection}{
  title={Handbook of the Geometry of Banach Spaces. Vol. I},
  editor={Johnson, W. B.},
  editor={Lindenstrauss, J.},
  publisher={North-Holland Publishing Co.},
  place={Amsterdam},
  date={2001},
}

\bib{JohnsonSchechtman2009}{article}{
  author={Johnson, W. B.},
  author={Schechtman, G.},
  title={Diamond graphs and super-reflexivity},
  journal={J. Topol. Anal.},
  fjournal={Journal of Topology and Analysis},
  volume={1},
  year={2009},
  number={2},
  pages={177--189},
  issn={1793-5253},
  mrclass={52C99 (46B10)},
  mrnumber={2541760 (2010k:52031)},
  url={http://dx.doi.org/10.1142/S1793525309000114},
}

\bib{KaltonLancien2008}{article}{
  author={Kalton, N. J.},
  author={Lancien, G.},
  title={Best constants for Lipschitz embeddings of metric spaces into $\co$},
  journal={Fund. Math.},
  volume={199},
  year={2008},
  number={3}
  pages={249--272},
}

\bib{Kloeckner2014}{article}{
  author={Kloeckner, B. R.},
  title={Yet another short proof of {B}ourgain's distortion estimate for embedding of trees into uniformly convex {B}anach spaces},
  journal={Israel J. Math.},
  fjournal={Israel Journal of Mathematics},
  volume={200},
  year={2014},
  number={1},
  pages={419--422},
  issn={0021-2172},
  mrclass={46B25},
  mrnumber={3219585},
  url={http://dx.doi.org/10.1007/s11856-014-0024-4},
}

\bib{Kutzarova1991}{article}{
  author={Kutzarova, D.},
  title={{$k$}-{$\beta $} and {$k$}-nearly uniformly convex {B}anach spaces},
  journal={J. Math. Anal. Appl.},
  fjournal={Journal of Mathematical Analysis and Applications},
  volume={162},
  year={1991},
  number={2},
  pages={322--338},
  issn={0022-247X},
  coden={JMANAK},
  mrclass={46B04 (46B03 46B20)},
  mrnumber={1137623 (93b:46018)},
  mrreviewer={Yves Raynaud},
  url={http://dx.doi.org/10.1016/0022-247X(91)90153-Q},
}

\bib{Kutzarova1990}{article}{
  author={Kutzarova, D.},
  title={An isomorphic characterization of property {$(\beta )$} of {R}olewicz},
  journal={Note Mat.},
  fjournal={Note di Matematica},
  volume={10},
  year={1990},
  number={2},
  pages={347--354},
  issn={1123-2536},
  mrclass={46B20},
  mrnumber={1204212 (94a:46020)},
  mrreviewer={S. Rolewicz},
}

\bib{LeeNaorPeres2009}{article}{
   author={Lee, J. R.},
   author={Naor, A.},
    AUTHOR = {Peres, Y.},
     TITLE = {Trees and {M}arkov convexity},
   JOURNAL = {Geom. Funct. Anal.},
  FJOURNAL = {Geometric and Functional Analysis},
    VOLUME = {18},
      YEAR = {2009},
    NUMBER = {5},
     PAGES = {1609--1659},
      ISSN = {1016-443X},
     CODEN = {GFANFB},
   MRCLASS = {05C05 (05C12 51F99 60B99)},
  MRNUMBER = {2481738 (2010e:05065)},
       URL = {http://dx.doi.org/10.1007/s00039-008-0689-0},
}

\bib{LimaR2012}{article}{
  author={Lima, V.},
  author={Randrianarivony, N. L.},
  title={Property $(\beta)$ and uniform quotient maps},
  journal={Israel J. Math.},
  volume={192},
  year={2012},
  pages={311--323},
}

\bib{LinialMagenSaks1998}{article}{
  author={Linial, N.},
  author={Magen, A.},
  author={Saks, M.},
  title={Low distortion {E}uclidean embeddings of trees},
  journal={Israel J. Math.},
  fjournal={Israel Journal of Mathematics},
  volume={106},
  year={1998},
  pages={339--348},
  issn={0021-2172},
  url={http://dx.doi.org/10.1007/BF02773475},
}

\bib{Matousek1999}{article}{
  author={Matou{\v {s}}ek, J.},
  title={On embedding trees into uniformly convex Banach spaces},
  journal={Israel J. Math.},
  volume={114},
  year={1999},
  pages={221--237},
}

\bib{MendelNaor2004}{article}{
  author={Mendel, M.},
  author={Naor, A.},
  title={Euclidean quotients of finite metric spaces},
  journal={Adv. Math.},
  volume={189},
  year={2004},
  pages={451--494},
}

\bib{MendelNaor2013}{article}{
  author={Mendel, M.},
  author={Naor, A.},
     TITLE = {Markov convexity and local rigidity of distorted metrics},
   JOURNAL = {J. Eur. Math. Soc. (JEMS)},
    VOLUME = {15},
      YEAR = {2013},
    NUMBER = {1},
     PAGES = {287--337},
     }

\bib{Milman1971}{article}{
  author={Milman, V. D.},
  title={Geometric theory of Banach spaces. II. Geometry of the unit ball},
  language={Russian},
  journal={Uspehi Mat. Nauk},
  volume={26},
  date={1971},
  pages={73--149},
  note={English translation: Russian Math. Surveys {\bf 26} (1971), 79--163},
}

\bib{Ostrovskii2012}{article}{
  author={Ostrovskii, M. I.},
  title={Embeddability of locally finite metric spaces into {B}anach spaces is finitely determined},
  journal={Proc. Amer. Math. Soc.},
  volume={140},
  year={2012},
  number={8},
  pages={2721--2730},
}

\bib{Ostrovskii2014AGMS}{article}{
    AUTHOR = {Ostrovskii, M. I.},
     TITLE = {Metric characterizations of superreflexivity in terms of word
              hyperbolic groups and finite graphs},
   JOURNAL = {Anal. Geom. Metr. Spaces},
  FJOURNAL = {Analysis and Geometry in Metric Spaces},
    VOLUME = {2},
      YEAR = {2014},
     PAGES = {154--168},
      ISSN = {2299-3274},
   MRCLASS = {46B85 (05C12 20F67 46B07)},
  MRNUMBER = {3210894},
MRREVIEWER = {Leonid V. Kovalev},
       URL = {http://dx.doi.org/10.2478/agms-2014-0005},
}

\bib{Pisier1975}{article}{
  author={Pisier, G.},
  title={Martingales with values in uniformly convex spaces},
  journal={Israel J. Math.},
  volume={20},
  date={1975},
  pages={326--350},
}

\bib{Rolewicz1987}{article}{
  author={Rolewicz, S.},
  title={On $\Delta$ uniform convexity and drop property},
  journal={Studia Math.},
  volume={87},
  year={1987},
  pages={181--191},
}

\bib{Zhang2015}{article}{
  author={Zhang, S.},
  title={Coarse quotient mappings between metric spaces},
  journal={Israel J. Math.},
  volume={},
  date={},
  pages={to appear, arXiv:1403.1934 (2014), 14 pages},
}

\end{biblist}
\end{bibsection}
\end{document}